\documentclass{siamltex}
\usepackage{lineno,hyperref}
\usepackage{multirow}
\usepackage[english]{babel}
\usepackage{amsmath}
\usepackage{bm}
\usepackage[matha,mathx]{mathabx}
\usepackage{mathrsfs}
\usepackage{amssymb}
\usepackage{algpseudocode} 
\usepackage{algorithm}
\usepackage{graphicx}
\usepackage{amsmath}
\usepackage{color} 
\usepackage{caption}
\usepackage{enumitem} 

\makeatletter
\newcommand*{\rom}[1]{\expandafter\@slowromancap\romannumeral #1@}
\makeatother

\usepackage[table]{xcolor}
\definecolor{lightgray}{gray}{0.9}
\usepackage{arydshln}

\newtheorem{remark}{Remark}

\usepackage{here}
\usepackage{graphicx} 

\DeclareMathAlphabet{\mathpzc}{OT1}{pzc}{m}{it}
\title{Higher order multi-dimension reduction methods via the Einstein product}

\author{ A. Zahir \thanks{The UM6P Vanguard Center, Mohammed VI Polytechnic University, Green 
City, Morocco.}
\and K. Jbilou\footnotemark[1] \thanks{Université du Littoral Cote d'Opale, LMPA, 50 rue F. Buisson, 62228 Calais-Cedex, France.}
\and A. Ratnani\footnotemark[1]}

\begin{document}
\maketitle

\begin{abstract}
This paper explores the extension of dimension reduction (DR) techniques to the multi-dimension case by using the Einstein product. Our focus lies on graph-based methods, encompassing both linear and nonlinear approaches, within both supervised and unsupervised learning paradigms. Additionally, we investigate variants such as repulsion graphs and kernel methods for linear approaches. Furthermore, we present two generalizations for each method, based on single or multiple weights. We demonstrate the straightforward nature of these generalizations and provide theoretical insights. Numerical experiments are conducted, and results are compared with original methods, highlighting the efficiency of our proposed methods, particularly in handling high-dimensional data such as color images.
\end{abstract}


\section{Introduction}
In today's data-driven world, where amounts of information are collected and analyzed, the ability to simplify and interpret data has never been more critical. The task is particularly evident in the field of data science and machine learning \cite{webb2011statistical}, where the curse of dimensionality is a major obstacle. \emph{Dimension reduction} techniques aim to address this issue by projecting high-dimensional data onto a lower-dimensional space while preserving the underlying structure of the data. These methods have been proven to be quite efficient in revealing hidden structures and patterns.

The landscape of DR is quite rich, with a wide range of methods, from linear to non-linear \cite{lee2007nonlinear}, supervised to unsupervised, and versions of these methods that incorporate \emph{repulsion-based principles, or kernels}.
We find as an example, Principal Component Analysis (PCA)\cite{PCA}, Locality Preserving Projections (LPP)\cite{he2003locality}, Orthogonal Neighborhood Preserving Projections (ONPP)\cite{kokiopoulou2005orthogonal,kokiopoulou2007orthogonal}, Neighborhood Preserving Projections (NPP)\cite{kokiopoulou2005orthogonal}, Laplacian Eigenmap (LE)\cite{belkin2003laplacian}, Locally Linear Embedding (LLE)\cite{roweis2000nonlinear}... Each of these techniques has been subject to extensive research and applications, offering insights into data structures that are often hidden in high-dimensional spaces, these methods can also be seen as an \emph{optimization problem of trace}, with some constraints\cite{saul2006spectral}.

Current approaches often require transforming the \emph{multi-dimensional data}, such as images \cite{he2005face,belhumeur1997eigenfaces,turk1991face,liu2002face,yang2001face}, into a matrix, into flattened (vectorized) forms before analysis. This process, while it's fast, however, can be problematic, as it may lead to loss of inherent structure and relational information within the data.

This paper proposes a novel approach to generalize dimensional reduction techniques, employing the \emph{Einstein product}, a tool in \emph{tensor} algebra, which is the natural extension of the usual matrix product. By reformulating the operations of both linear and non-linear methods in the context of tensor operations, the generalization maintains the multi-dimensional integrity of complex datasets. This approach circumvents the need for vectorization, preserving the rich intrinsic structure of the data.\\
Our contribution lies in not only proposing a generalized framework for dimensional reduction, but also in demonstrating its effectiveness through empirical studies. We show that the proposed methods, outperform or at least are the same as their matrix-based counterparts, while preserving the integrity of the data.

This paper is organized as follows. Firstly, we will talk in Section \ref{sec:GB-2D} about the methods in the matrix case, then, in Section \ref{sec:Prop}, we will introduce the mathematical background of tensors, and the Einstein product. Next in Section \ref{sec:DR:3D}, we will introduce the different methods, and the generalization of these methods using the Einstein product. Following that, Section \ref{sec:Variants} is dedicated to presenting variants of these techniques. Subsequently, we will present the numerical experiments and the results in Section \ref{sec:Exp}. Lastly, we offer some concluding remarks and suggestion of future work in Section \ref{sec:Cnc}.

\section{Dimension reduction methods in matrix case}\label{sec:GB-2D}
Given a set of $n$ data points $\mathbf{x}_1,\ldots,\mathbf{x}_n \in \mathbb{R}^m$ and a set of $n$ corresponding points $\mathbf{y}_1,\ldots,\mathbf{y}_n \in \mathbb{R}^d$, denote the data matrix $X=\left[\mathbf{x}_{1}, \cdots, \mathbf{x}_{n}\right] \in \mathbb{R}^{m \times n} $ and the low-dimensional matrix $ Y= \left[\mathbf{y}_{1}, \cdots, \mathbf{y}_{n}\right] \in \mathbb{R}^{d \times n}$. The objective is to find a mapping $\Phi: \mathbb{R}^{m} \longrightarrow \mathbb{R}^{d} , \phi\left(\mathbf{x}_i\right)=\mathbf{y}_i, \quad i=1, \cdots, n$. The mapping is either non-linear $Y=\Phi(X)$, or linear $Y=V^\top X$, in the latter case, it reduces to find the projection matrix $V \in \mathbb{R}^{m \times d}$.

We denote the similarity matrix of a graph by $W \in \mathbb{R}^{n \times n}$, the degree matrix by $D$, and the Laplacian matrix by $L=D-W$. 
For the sake of simplifications, we will define some new matrices
\begin{equation*}
\begin{aligned}
& 
L_n= D^{-1 / 2} L D^{-1 / 2} \text { , }
\widehat{W}=D^{-1 / 2} W D^{-1 / 2} \text { , }
M=(I_n-W^\top )(I_n-W) \text { , } \\
& 
\widehat{X}=X D^{1 / 2} \text { , } 
\widehat{Y}=Y D^{1 / 2} \text { , } 
H=I_n-\dfrac{1}{n}\mathbf{1}\mathbf{1}^\top,
\end{aligned}
\end{equation*}
where $H$ is the centering matrix, and $\mathbf{1}=\left(1,\ldots, 1 \right)^\top \in \mathbb{R}^n.$\\
The usual loss functions used are defined as follows
\begin{eqnarray} 
\phi_1(Y)&:=&\dfrac{1}{2} \sum_{i, j=1}^{n} W_{i j}\left\|\mathbf{y}_{i}-\mathbf{y}_{j}\right\|_{2}^{2}=\operatorname{Tr}\left[Y L Y^{H}\right], \label{eq : loss1}\\
\phi_2(Y)&:=&\sum_{i}\left\|\mathbf{y}_{i}-\sum_{j} W_{i j} \mathbf{y}_{j}\right\|_{2}^{2}=\operatorname{Tr}\left[YMY^{H}\right],\label{eq : loss2}\\
\Phi_3(Y)&:=&\sum_{i}\left\|\mathbf{y}_{i}-\dfrac{1}{n}\sum_{j} \mathbf{y}_{j}\right\|_{2}^{2}=\operatorname{Tr}\left[Y (I-\dfrac{1}{n}\mathbf{1}\mathbf{1}^\top ) Y^{H} \right]. \label{eq : loss3}
\end{eqnarray}
Equations \eqref{eq : loss1}, and \eqref{eq : loss3} preserve the locality, i.e., the point and its representation stay close, while Equation~\eqref{eq : loss2} preserves the local geometry, i.e., the representation point can be written as a linear combination of its neighbours.

For simplicity, we will refer to the $d$ eigenvectors of a matrix corresponding to the largest and smallest eigenvalues, respectively, as the largest and smallest $d$ eigenvectors of a matrix. The same terminology applies to the left or right singular vectors.
Table \ref{table1} summarizes the various dimension reduction methods, their corresponding optimization problems and the solutions.
\begin{table}[H]
\centering
\begin{tabular}{|>{\centering\arraybackslash}p{40mm}|>{\centering\arraybackslash}p{21mm}|p{20mm}|>{\centering\arraybackslash}p{46mm}|}
\rowcolor{lightgray}
\hline \textbf{Method} & \textbf{Loss \newline function} & \textbf{Constraint} & \textbf{Solution} \\
\hline
\multicolumn{4}{|l|}{\textbf{Linear methods}} \\
\hline Principal component analysis\cite{PCA}. & Maximize Equation~\eqref{eq : loss3}. & $V V^\top=I$ & Largest $d$ left singular vectors of $XH$. \\
\hline
Locality Preserving Projections. \cite{he2003locality} & Minimize Equation~\eqref{eq : loss1} & $Y D Y^\top=I$ & Solution of $\widehat{X}(I_n-\widehat{W}) \widehat{X}^\top u_i= \lambda_i \widehat{X} \widehat{X}^\top u_i$. \\
\hline
Orthogonal Locality Preserving Projections.\cite{kokiopoulou2005orthogonal,kokiopoulou2007orthogonal} & Minimize Equation~\eqref{eq : loss1} & $V V^\top=I$ & Smallest $d$ eigenvectors of $X L X^\top$. \\ 
\hline
Orthogonal Neighborhood Preserving Projections.\cite{kokiopoulou2005orthogonal,kokiopoulou2007orthogonal} & Minimize Equation~\eqref{eq : loss2} & $V V^\top=I$ & Smallest $d$ eigenvectors of $X M X^\top$. \\
\hline
Neighborhood Preserving Projections.\cite{kokiopoulou2005orthogonal} & Minimize Equation~\eqref{eq : loss2} & $Y Y^\top=I$ & Sol of $XMX^\top u_i=\lambda_i X X^\top u_i$ \\
\hline
\multicolumn{4}{|l|}{\textbf{Non-Linear methods}} \\
\hline
Locally Linear Embedding.\cite{roweis2000nonlinear} & Minimize Equation~\eqref{eq : loss2} & $Y Y^\top=I$ & Smallest $d$ Eigenvectors of $M$.\\
\hline
Laplacian Eigenmap.\cite{belkin2003laplacian} & Minimize Equation~\eqref{eq : loss1} & $Y D Y^\top=I$ & Solution of $L u_i = \lambda_i D u_i$. \\ 
\hline
\end{tabular}
\caption{Objective functions and constraints employed in various dimension reduction methods along with corresponding solutions.}\label{table1}
\end{table}
\noindent Notice that the smallest eigenvalue is disregarded in the solutions, thus, the second to the $d+1$ eigenvectors are taken.
The graph based methods are quite similar, each one tries to give an accurate representation of the data while preserving a desired property. The solution of the optimization problem is given by the eigenvectors, or the singular vectors.

Next, we will introduce notations related to the tensor theory (Einstein product) as well as some properties that guarantee the proposed generalization.
\section{The Einstein product and its properties}\label{sec:Prop}
Let $\mathbf{I}=\{I_1,\ldots,I_N\}$ and $\mathbf{J}=\{J_1,\ldots,J_M\}$ be two multi-indices, and $\mathbf{i}=\{i_1,\ldots,i_N\}$ and $\mathbf{j}=\{j_1,\ldots,j_M\}$ be two indices. The index mapping function $\operatorname{ivec}(\mathbf{i}, \mathbf{I})=i_1+\sum_{k=2}^N\left(i_k-1\right) \prod_{l=1}^{k-1} I_l$ that maps the multi-index $\mathbf{i}$ to the corresponding index in the vectorized form of a tensor of size $I_1 \times \ldots \times I_N$. 
The unfolding, also known also as flattening or matricizaion, is a function $\Psi: \mathbb{R}^{I_1 \times I_2 \times \cdots \times I_N \times J_1 \times J_2 \times \cdots \times J_M} \longrightarrow \mathbb{R}^{|\mathbf{I}| \times|\mathbf{J}|}, \; \mathcal{A} \mapsto A$ with $A_{i j}=\mathcal{A}_{i_1 i_2 \ldots i_N j_1 j_2 \ldots j_M}$, that maps a tensor into a matrix, with the subscripts $i=\operatorname{ivec}(\mathbf{i}, \mathbf{I}),$ and $j=\operatorname{ivec}(\mathbf{j}, \mathbf{J})$. The mapping $\Psi$ is a linear isomorphism, and its inverse is denoted by $\Psi^{-1}$. It would generalize some concepts of the matrix theory more easily.

The frontal slice of the $N$-order tensor $\mathcal{A}\in \mathbb{R}^{I_1\times \ldots \times I_N}$, denoted by $\mathcal{A}^{(i)}$ is the tensor $\mathcal{A}_{:,\ldots,:i}$ (the last mode is fixed to $i$). A tensor $\mathcal{A} \in \mathbb{R}^{I_1\times \ldots \times I_N \times J_1\times \ldots \times J_M}$ is called even if $N=M$ and square if $I_i=J_i$ for all $i=1,\ldots,N$ \cite{qi2017tensor}.\\

\begin{definition}[m-mode product]\cite{Kolda2009}
Let $\mathcal{X} \in \mathbb{R}^{I_1 \times \ldots \times I_M}$, and $U \in \mathbb{R}^{J\times I_m}$, the m-mode (matrix) product of $\mathcal{X}$ and $U$ is a tensor of size $I_1 \times \ldots I_{m-1} \times J \times I_{m+1} \ldots \times I_M$, with element-wise
\begin{equation}
(\mathcal{X} \times_m U)_{i_1 \ldots i_{m-1}ji_{m+1} \ldots i_M}=\sum_{i_m=1}^{I_m} U_{j i_m} \mathcal{X}_{i_1 \ldots i_M}.
\end{equation}
\end{definition}
\medskip
\begin{definition}[Einstein product]\cite{Brazell2013}
Let $\mathcal{X} \in \mathbb{R}^{I_1 \times \ldots \times I_M \times K_1 \times \ldots \times K_N}$ and \\ $\mathcal{Y} \in \mathbb{R}^{K_1 \times \ldots \times K_N \times J_1 \times \ldots \times J_M}$. The Einstein product of $\mathcal{X}$ and $\mathcal{Y}$ is the tensor of size $ \mathbb{R}^{I_1 \times \ldots \times I_M \times J_1 \times \ldots \times J_M}$ whose elements are defined by
\begin{equation}
\left(\mathcal{X} *_N \mathcal{Y}\right)_{i_1 \ldots i_M j_1 \ldots j_M}=\sum_{k_1 \ldots k_N} \mathcal{X}_{i_1 \ldots i_M k_1 \ldots k_N} \mathcal{Y}_{k_1 \ldots k_N j_1 \ldots j_M}.
\end{equation}
\end{definition}

\medskip
Next, we have some definitions related to the Einstein product. 
\begin{definition}
\begin{itemize}
\item Let $\mathcal{A}\in \mathbb{R}^{I_1\times \ldots \times I_N\times J_1\times \ldots \times J_M}$, then the transpose tensor \cite{qi2017tensor} of $\mathcal{A}$ denoted by $\mathcal{A}^\top$ is the tensor of size $J_1\times \ldots \times J_M \times I_1\times \ldots \times I_N$ whose entries defined by $(\mathcal{A}^\top )_{j_1\dots j_M i_1\dots i_N}=\mathcal{A}_{i_1\dots i_N j_1\dots j_M}$. 
\item $\mathcal{A}$ is a diagonal tensor if all of its entries are zero except for those on its diagonal, denoted as $(\mathcal{A})_{i_1\dots i_N i_1\dots i_N}$, for all $ 1\leq i_r \leq \min(I_r,J_r), \; 1\leq r \leq N$.
\item The identity tensor denoted by $\mathcal{I}_N\in \mathbb{R}^{I_1\times \ldots \times I_N \times I_1 \times \ldots \times I_N}$ is a diagonal tensor with only ones on its diagonal. 
\item A square tensor $\mathcal{A} \in \mathbb{R}^{I_1\times \ldots \times I_N \times I_1 \times \ldots \times I_N}$ is called symmetric if $\mathcal{A}^\top = \mathcal{A}$.
\end{itemize}
\end{definition}
\medskip
\begin{remark}
In case of no confusion, The identity tensor will be denoted simply $\mathcal{I}$.
\end{remark}
\medskip
\begin{definition}
The inner product of tensors $\mathcal{X}, \mathcal{Y} \in \mathbb{R}^{I_1 \times \ldots \times I_N}$ is defined by 
\begin{equation}
\langle \mathcal{X}, \mathcal{Y}\rangle=\sum_{i_1,\ldots,i_N} \mathcal{X}_{i_1 i_2 \ldots i_N} \mathcal{X}_{i_1 i_2 \ldots i_N}.
\end{equation}
The inner product induces The Frobenius norm as follows
\begin{equation} 
\|\mathcal{X}\|_{F}=\sqrt{\langle \mathcal{X}, \mathcal{X}\rangle}. 
\end{equation}
\end{definition}

\medskip
\begin{definition}
A square $2N$-order tensor $\mathcal{A}$ in invertible (non-singular) if there is a tensor denoted by $\mathcal{A}^{-1}$ of same size such that $\mathcal{A} *_N \mathcal{A}^{-1}=\mathcal{A}^{-1} *_N \mathcal{A} =\mathcal{I}_N$. It is unitary if $\mathcal{A}^\top *_N \mathcal{A}=\mathcal{A} *_N \mathcal{A}^\top =\mathcal{I}_N$.
It is positive semi-definite if $\langle \mathcal{X},\mathcal{A} *_N \mathcal{X}\rangle \geq 0$ for all non-zero $\mathcal{X} \in \mathbb{R}^{I_1 \times \ldots \times I_N}$.
It is positive definite if the inequality is strict.
\end{definition}

\medskip
An important relationship that is easy to prove is the stability of the Frobenius norm under the Einstein product with a unitary tensor.
\medskip
\begin{proposition}
Let $\mathcal{X} \in \mathbb{R}^{I_1 \times \ldots \times I_M \times J_1 \times \ldots \times J_N}$ and $\mathcal{U} \in \mathbb{R}^{I_1 \times \ldots \times I_M \times I_1 \times \ldots \times I_M}$ be a unitary tensor. Then
\begin{equation}
 \|\mathcal{U} *_M \mathcal{X}\|_{F}=\|\mathcal{X}\|_{F}.
\end{equation}
\end{proposition}
The proof is straightforward using the inner product definition.
\medskip
\begin{proposition}\cite{Wang2019}
Let $\mathcal{X}, \mathcal{Y} \in \mathbb{R}^{I_1 \times \ldots \times I_N \times J_1 \times \ldots \times J_M}$. We have
\begin{equation}
\label{prop:trace_conjuguate}
\begin{aligned}
\langle\mathcal{X}, \mathcal{Y}\rangle&=\operatorname{Tr}\left(\mathcal{X}^\top *_N \mathcal{Y}\right)\\
&= \operatorname{Tr}(\mathcal{Y} *_M \mathcal{X}^\top ).
\end{aligned}
\end{equation}
\end{proposition}

\begin{proposition}\cite{Einstein_inverse}
Given tensors $\mathcal{X} \in \mathbb{R}^{I_1 \times \ldots \times I_N \times K_1 \times \ldots \times K_N}$, $\mathcal{Y} \in \mathbb{R}^{K_1 \times \ldots \times K_N \times J_1 \times \ldots \times J_M}$. We have
\begin{equation}
 \label{eq:Einstein_transpose}
 \left(\mathcal{X} *_N \mathcal{Y}\right)^\top=\mathcal{Y}^\top *_N \mathcal{X}^\top.
\end{equation}
\end{proposition}
The isomorphism $\phi$ has some properties that would be useful in the following.
\medskip
\begin{proposition}\cite{Wang2019} \label{prop:automorphisme}
Given the tensors $\mathcal{X}$ and $\mathcal{Y}$ of appropriate size then, we have 
$\phi$ is a multiplicative morphism with respect the Einstein product, i.e., $\Psi(\mathcal{X} *_N \mathcal{Y})=\Psi(\mathcal{X}) \Psi(\mathcal{Y})$.
\end{proposition}\medskip

It allows us to prove the Einstein Tensor Spectral Theorem.
\medskip
\begin{theorem}[Einstein Tensor Spectral Theorem]\label{thm:Spectral}
A symmetric tensor is diagonalizable via the Einstein product.
\end{theorem}
\medskip
\begin{proof}
The proof is using the isomorphism and its properties \ref{prop:automorphisme}.

\medskip
\noindent Let $\mathcal{X}$ be a symmetric tensor of size $ \mathbb{R}^{I_1 \times \ldots \times I_N \times I_1 \times \ldots \times I_N}$, then $\Psi(\mathcal{X})$ is symmetric, and by the spectral theorem, there exists an orthogonal matrix $U$ such that $U^\top \Psi(\mathcal{X}) U=\Lambda$, where $\Lambda$ is a diagonal matrix.\\
Then, $\Psi(\mathcal{X})=U \Lambda U^\top$, and $\mathcal{X}=\phi^{-1}(U \Lambda U^\top)=\Psi^{-1}(U) *_N \Psi^{-1}(\Lambda) *_N \Psi^{-1}(U^\top)=\Psi^{-1}(U) *_N \Psi^{-1}(\Lambda) *_N \Psi^{-1}(U)^\top$, with $\Psi^{-1}(U)$ is a unitary, and $\Psi^{-1}(\Lambda)$ is diagonal tensor.
\end{proof}

\medskip
The cyclic property of the trace with Einstein product is also verified, which is needed in the sequel.
\medskip
\begin{proposition}[Cyclic property of the trace]\label{pro:cyclic_trace}
Let $\mathcal{X} \in \mathbb{R}^{I_1 \times \ldots \times I_M \times K_1 \times \ldots \times K_N}$, $
\mathcal{Y} \in \mathbb{R}^{K_1 \times \ldots \times K_N \times I_1 \times \ldots \times I_M}$, and $ \mathcal{Z} \in \mathbb{R}^{K_1 \times \ldots \times K_N \times K_1 \times \ldots \times K_N} $. We have
\begin{equation}
\label{eq:Einstein_cyclic}
\operatorname{Tr}\left(\mathcal{X} *_N \mathcal{Z} *_N \mathcal{Y}\right)=\operatorname{Tr}\left(\mathcal{Y} *_M \mathcal{X} *_N \mathcal{Z} \right).
\end{equation}
\end{proposition}
\medskip

\begin{theorem}\cite{Einstein_inverse}
Let $\mathcal{X} \in \mathbb{R}^{I_1 \times \ldots \times I_M \times K_1 \times \ldots \times K_N}$, the Einstein singular value decomposition (E-SVD) of $\mathcal{X}$ is defined by
\begin{equation}
\mathcal{X}=\mathcal{U} *_M \mathcal{S} *_N \mathcal{V}^\top,
\end{equation}
where $\mathcal{U} \in \mathbb{R}^{I_1 \times \ldots \times I_M \times I_1 \times \ldots \times I_M}, \mathcal{S} \in \mathbb{R}^{I_1 \times \ldots \times I_M \times K_1 \times \ldots \times K_N}, \mathcal{V} \in \mathbb{R}^{K_1 \times \ldots \times K_N \times K_1 \times \ldots \times K_N}$ with the following properties\\
$\mathcal{U}$ and $\mathcal{V}$ are orthogonal, where $\mathcal{U}_{: \ldots : i_1 \ldots i_M}, \mathcal{V}_{: \ldots : j_1 \ldots j_N}$ are the left and right singular tensors of $\mathcal{X}$, respectively.
If $N<M$, then \[\mathcal{S}_{i_1 \ldots i_M k_1 \ldots k_N}=\begin{cases} d_{k_1 \ldots k_N} & \text{if } (i_1,\ldots,i_N)=(k_1,\ldots,k_N) \text{ and } (i_{N+1},\ldots,i_M)=(1,\ldots,1) \\ 0 & \text{otherwise.} \end{cases}\].\\
If $N=M$, then $\mathcal{S}_{i_1 \ldots i_M k_1 \ldots k_N}=\begin{cases} d_{k_1 \ldots k_N} & \text{if } (i_1,\ldots,i_M)=(k_1,\ldots,k_M) \\0 & \text{otherwise.}\end{cases}$\\
The numbers $d_{k_1 \ldots k_N}$ are the singular values of $\mathcal{X}$ with the decreasing order
\[d_{{1,\ldots,1}} \geq d_{2,1,\ldots,1}\geq \ldots \geq d_{\widehat{K}_1,1,\ldots,1} \\ \geq d_{1,2,\ldots,1} \geq \ldots \geq d_{1,\widehat{K}_2,\ldots,1} \geq \ldots \geq d_{\widehat{K}_1,\ldots,\widehat{K}_P} \geq 0,\]
with $P=\min(N,L),\; \widehat{K}_r=\min(I_r,K_r), \; r=1,\ldots,P$.
\end{theorem}

\medskip
 We define the eigenvalues and eigen-tensors of a tensor with the following.
 \medskip
\begin{definition}\cite{wang2022generalized}
Let a square $2N$-th order tensors $\mathcal{A},\mathcal{B} \in \mathbb{R}^{I_1 \times \ldots \times I_N \times I_1 \times \ldots \times I_N}$, then
\begin{itemize}
\item \textbf{Tensor Eigenvalue problem:} If there is a non null $\mathcal{X} \in \mathbb{R}^{I_1 \times \ldots \times I_N}$, and $\lambda \in \mathbb{R}$ such that $\mathcal{A} *_{N} \mathcal{X}=\lambda \mathcal{X}$, then $\mathcal{X}$ is called an eigen-tensor of $\mathcal{A}$, and $\lambda$ is the corresponding eigenvalue.
\item \textbf{Tensor generalized Eigenvalue problem:} If there is a non null $\mathcal{X} \in \mathbb{R}^{I_1 \times \ldots \times I_N}$, and $\lambda \in \mathbb{R}$ such that $\mathcal{A} *_{N} \mathcal{X}=\lambda \mathcal{B} *_N \mathcal{X}$, then $\mathcal{X}$ is called an eigen-tensor of the pair $\{ \mathcal{A},\mathcal{B} \}$, and $\lambda$ is the corresponding eigenvalue.
\end{itemize}
\end{definition}\medskip
\begin{remark}
If $N=1$, the two definitions above coincide with the eigenvalue and generalized eigenvalue problems, respectively.
\end{remark}

\medskip
We can also show a relationship between the singular values and the eigenvalues of a tensor.
\medskip
\begin{proposition} 
Let the E-SVD of $\mathcal{X} \in \mathbb{R}^{I_1 \times \ldots \times I_M \times K_1 \times \ldots \times K_N}$, defined as $\mathcal{X}=\mathcal{U} *_M \mathcal{S} *_N \mathcal{V}^\top$, then
\begin{itemize}
\item The eigenvalues of $\mathcal{X} *_N \mathcal{X}^\top$ and $\mathcal{X}^\top \times_M \mathcal{X}$ are the squared singular values of $\mathcal{X}$.
\item The eigen-tensors of $\mathcal{X} *_N \mathcal{X}^\top$ are the left singular tensors of $\mathcal{X}$.
\item The eigen-tensors of $\mathcal{X}^\top *_M \mathcal{X}$ are the right singular tensors of $\mathcal{X}$.
\end{itemize}
\end{proposition}

\medskip
The proof is straightforward.

To simplify matters, we'll denote the $d$ eigen-tensors of a tensor, associated with the smallest eigenvalues, as the smallest $d$ eigen-tensors. Similarly, we will apply the same principle to the largest $d$ eigen-tensors. This terminology also extends to the left or right singular tensors.

\medskip
\begin{remark}
To generalize the notion of left and right inverse for a non-square tensors. It is called left or right $\Psi$-invertible if $\Psi(\mathcal{A})$ is left or right invertible, respectively. In case of confusion, we will denote $\Psi$ by $\Psi_j$ to represent the transformation of tensor $\mathcal{A} \in \mathbb{R}^{I_1 \times \ldots \times I_N}$ to a matrix $ \mathbb{R}^{(\prod_{k=1}^j I_k) \times ((\prod_{k=j+1}^N I_k)}$.
\end{remark}

\medskip
\begin{proposition}\label{prop:symmetric_XMXt}
\begin{enumerate}[label=\arabic*.]
\item A square $2N$-th order symmetric $\mathcal{X}$ is positive semi-definite tensor, definite tensor, respectively, if and only if there is a tensor, an invertible tensor, respectively, $\mathcal{B}$ of same size such that $\mathcal{X}=\mathcal{B} *_N \mathcal{B}^\top$.
\item Let a tensor $\mathcal{X} \in \mathbb{R}^{I_1 \times \ldots \times I_N}$ of order N, with its transpose in $ \mathbb{R}^{I_{N-j+1} \times \ldots \times I_N \times I_1 \times \ldots I_{N-j}}$ then $\mathcal{X} *_j \mathcal{X}^\top$ is positive semi-definite for any $1 \leq j \leq N$.\\
Let a tensor $\mathcal{X} \in \mathbb{R}^{I_1 \times \ldots \times I_N}$ of order N, let $1 \leq j \leq N$ such that the tensor is $\Psi_j-$ invertible, with its transpose in $ \mathbb{R}^{I_{N-j+1} \times \ldots \times I_N \times I_1 \times \ldots I_{N-j}}$, then $\mathcal{X} *_j \mathcal{X}^\top$ is positive definite.
\item The eigenvalues of a square symmetric tensor are real.
\item Let a symmetric matrix $M \in \mathbb{R}^{K \times K}$ and $\mathcal{X} \in \mathbb{R}^{I_1 \times \ldots \times I_N \times K}$, with its transpose in $ \mathbb{R}^{K \times I_1 \times \ldots \times I_N}$ then $ \mathcal{X} \times_{N+1} M *_1 \mathcal{X}^\top$ is a symmetric.
\item If $M$ is positive semi-definite, then $\mathcal{X} \times_{N+1} M *_1 \mathcal{X}^\top$ is positive semi-definite.
\item If $M$ is positive definite, and $\mathcal{X}$ is $\Psi-$invertible, then $\mathcal{X} \times_{N+1} M *_1 \mathcal{X}^\top$ is positive definite.
\end{enumerate}
\end{proposition}
\medskip
\begin{proof} 
The proof of the first one is straightforward using \ref{prop:automorphisme}.\\
Let $\mathcal{Y} \in \mathbb{R}^{ I_1 \times \ldots I_{N-j}}$, then $\mathcal{Y}^\top *_{N-j} \mathcal{X} *_j \mathcal{X}^\top *_{N-j} \mathcal{Y}=\|\mathcal{X}^\top *_{N-j} \mathcal{Y}\|_F^2 \geq 0$.\\
The proof of the third is similar to the second one.\\
Let $\mathcal{A} \in \mathbb{R}^{I_1 \times \ldots \times I_N \times I_1 \times \ldots \times I_N}$ be a symmetric tensor and non-zero tensor $\mathcal{X} \in \mathbb{R}^{I_1 \times \ldots \times I_N}$ with $\mathcal{A} *_N \mathcal{X}=\lambda \mathcal{X}$, then
$\lambda^\top \mathcal{X}^\top =(\lambda \mathcal{X})^\top=(\mathcal{A} *_N \mathcal{X})^\top= \mathcal{X}^\top *_N \mathcal{A}^\top = \mathcal{X}^\top *_N\mathcal{A}=\lambda^\top \mathcal{X}^\top$, then $\lambda=\lambda^\top$, which completes the proof.\\
We have $\left( \mathcal{X} \times_{N+1} M *_1 \mathcal{X}^\top\right)^\top=(M \times_1 \mathcal{X}^\top)^\top *_1 \mathcal{X}^\top =(\mathcal{X} \times_{N+1} M^\top) *_1 \mathcal{X}^\top$, then conclude by the symmetry of $M$.\\
Let $M$ be a positive semi-definite matrix, then there exist a matrix $B$ such that $M=BB^\top$, then \[ 
\mathcal{X} \times_{N+1} M *_1 \mathcal{X}^\top =\mathcal{X} \times_{N+1} B B^\top *_1 \mathcal{X}^\top =\widehat{\mathcal{X}} *_1 \widehat{\mathcal{X}}^\top, \]
with $\widehat{\mathcal{X}}=\mathcal{X} \times_{N+1} B$, then the result follows.\\
The last proposition has a similar proof; Using the fact that $M$ is positive definite, then $B$ is invertible, and $\mathcal{X}$ is invertible, then $\widehat{\mathcal{X}}$ is $\Psi-$invertible, and the result follows.
\end{proof}

We present a property that relates the tensor generalized eigenvalue problem with the tensor eigenvalue problem.
\medskip
\begin{proposition}
Let the generalized eigenvalue problem $\mathcal{A} *_{N} \mathcal{X}=\lambda \mathcal{M} *_{N} \mathcal{X}$, with $\mathcal{A}, \mathcal{M}$ are a square $2N$-th order tensor, with $\mathcal{M}$ being invertible, then
$\widehat{\mathcal{X}}= \mathcal{M} *_{N} \mathcal{X}$ is a solution of the tensor eigen-problem
$\widehat{\mathcal{A}} *_{N} \widehat{\mathcal{X}}=\lambda \widehat{\mathcal{X}}$ with $\widehat{\mathcal{A}}=\mathcal{A} *_N \mathcal{M}^{-1}$.
\end{proposition}
\medskip
\begin{theorem}\label{thm:tr_pos_def}
Let a symmetric $\mathcal{X} \in \mathbb{R}^{I_1 \times \ldots \times I_M \times I_1 \times \ldots \times I_M}$, and $\mathcal{B}$ a positive definite tensor of same size, then
\[\min_{\substack{\mathcal{P} \in \mathbb{R}^{I_1 \times \ldots \times I_M \times d}\\ \mathcal{P}^\top *_M \mathcal{B} *_M \mathcal{P}=\mathcal{I}}} \operatorname{Tr}(\mathcal{P}^\top *_M \mathcal{X} *_M \mathcal{P}), \] is equivalent to solve the generalized eigenvalue problem $\mathcal{X} *_M \mathcal{P}=\lambda \mathcal{B} *_M \mathcal{P}$. 
\end{theorem}
\begin{proof}
Since $\Psi$ is an isomorphism, the problem is equivalent to minimize $\operatorname{Tr}(PXP^\top )$ with $P^\top B P=I$, 
$\Psi(\mathcal{P})=P, \Psi(\mathcal{X})=X, \Psi(\mathcal{B})=B$. We have $X$ symmetric and $B$ is positive definite. The solution of the equivalent problem is the $d$ smallest eigenvalues of $X$, using the fact $\Psi^{-1}$ is an isomorphism, we obtain the result.\\
A second proof without using the isomorphism property is the following.\\
Let the Lagrangian of the problem be
\begin{equation*}
\mathcal{L}(\mathcal{P},\Lambda):=\operatorname{Tr}(\mathcal{P}^\top *_M \mathcal{X} *_M \mathcal{P})-\operatorname{Tr}(\Lambda^\top *_M (\mathcal{P}^\top *_M \mathcal{B} *_M \mathcal{P}-\mathcal{I})),
\end{equation*}
with $\Lambda \in \mathbb{R}^{d \times d}$ the Lagrange multiplier. Using KKT conditions, we have
\begin{equation*}
\begin{aligned}
\dfrac{\partial \mathcal{L}}{\partial \mathcal{P}}&=2 \mathcal{X} *_M \mathcal{P}+2 \Lambda *_M \mathcal{P}=0\\
\implies &\mathcal{P}^\top *_M \mathcal{B} *_M \mathcal{P}-\mathcal{I}=0.
\end{aligned}
\end{equation*}
To compute the partial derivative with respect to $\mathcal{P}$, we introduce the functions $f_1(\mathcal{P})$ and $f_2(\mathcal{P})$, defined as follows
\[ f_1(\mathcal{P}) = \operatorname{Tr}(\mathcal{P}^\top *_M \mathcal{X} *_M \mathcal{P}), \]
\[ f_2(\mathcal{P}) = \operatorname{Tr}\left(\Lambda^\top *_M (\mathcal{P}^\top *_M \mathcal{B} *_M \mathcal{P} - \mathcal{I})\right). \]
Subsequently, we aim to determine the partial derivative.
\begin{equation*}
\begin{aligned}
f_1(\mathcal{P}+\varepsilon \mathcal{H})&=\operatorname{Tr}\left( \left(\mathcal{P}+\varepsilon \mathcal{H})^\top *_M \mathcal{X} *_M (\mathcal{P}+\varepsilon \mathcal{H}\right)\right)\\
&=\operatorname{Tr}\left(\mathcal{P}^\top *_M \mathcal{X} *_M \mathcal{P}+\varepsilon \mathcal{H}^\top *_M \mathcal{X} *_M \mathcal{P}+\varepsilon \mathcal{P}^\top *_M \mathcal{X} *_M \mathcal{H}\right)\\
&\qquad +\operatorname{Tr}(\varepsilon^2 \mathcal{H}^\top *_M \mathcal{X} *_M \mathcal{H})\\
&=\operatorname{Tr}(\mathcal{P}^\top *_M \mathcal{X} *_M \mathcal{P})+\varepsilon \operatorname{Tr}(\mathcal{H}^\top *_M \mathcal{X} *_M \mathcal{P})\\
&\qquad +\varepsilon \operatorname{Tr}(\mathcal{P}^\top *_M \mathcal{X} *_M \mathcal{H})+\varepsilon^2 \operatorname{Tr}(\mathcal{H}^\top *_M \mathcal{X} *_M \mathcal{H})\\
&=f_1(\mathcal{P})+\varepsilon \left[ \operatorname{Tr}(\mathcal{H}^\top *_M \mathcal{X} *_M \mathcal{P})+ \operatorname{Tr}(\mathcal{P}^\top *_M \mathcal{X} *_M \mathcal{H}) \right] +o(\varepsilon)\\
&=f_1(\mathcal{P})+\varepsilon \operatorname{Tr}\left(\mathcal{H}^\top *_M (\mathcal{X} +\mathcal{X}^\top )*_M \mathcal{P}\right) +o(\varepsilon).\\
\end{aligned}
\end{equation*}

Then, as $\mathcal{X}$ is symmetric, the partial derivative in the direction $\mathcal{H}$ is
\begin{equation*}
\begin{aligned}
\dfrac{\partial f_1}{\partial \mathcal{P}}(\mathcal{H})&=\lim_{\varepsilon \rightarrow 0} \dfrac{f_1(\mathcal{P}+\varepsilon \mathcal{H})-f_1(\mathcal{P})}{\varepsilon}\\
&=2\operatorname{Tr}(\mathcal{H}^\top *_M \mathcal{X}*_M \mathcal{P}).
\end{aligned}
\end{equation*}
It gives us the partial derivative of $f_1$ with respect to $\mathcal{P}$ as
\begin{equation*}
\dfrac{\partial f_1}{\partial \mathcal{P}}=2\mathcal{X}*_M \mathcal{P}.
\end{equation*}
For the second function, we have
\begin{equation*}
\begin{aligned}
f_2(\mathcal{P}+\varepsilon \mathcal{H})&=\operatorname{Tr}\left(\Lambda^\top *_M \left((\mathcal{P}+\varepsilon \mathcal{H})^\top *_M \mathcal{B} *_M (\mathcal{P}+\varepsilon \mathcal{H})-\mathcal{I}\right)\right)\\
&=\operatorname{Tr}\left(\Lambda^\top *_M (\mathcal{P}^\top *_M \mathcal{B} *_M \mathcal{P}-\mathcal{I}))++\varepsilon^2 \operatorname{Tr}(\Lambda^\top *_M (\mathcal{H}^\top *_M \mathcal{B} *_M \mathcal{H})\right)\\
+& \varepsilon \operatorname{Tr}\left(\Lambda^\top *_M (\mathcal{H}^\top *_M \mathcal{B} *_M \mathcal{P}+\mathcal{P}^\top *_M \mathcal{B} *_M \mathcal{H})\right)\\
&=f_2(\mathcal{P})+\varepsilon \operatorname{Tr}\left(\Lambda^\top *_M (\mathcal{H}^\top *_M \mathcal{B} *_M \mathcal{P}+\mathcal{P}^\top *_M \mathcal{B} *_M \mathcal{H})\right)+ O(\varepsilon)\\
&=f_2(\mathcal{P})+\varepsilon \operatorname{Tr}\left(\mathcal{H}^\top *_M \left[ \mathcal{B} *_M \mathcal{P} *_M \Lambda^\top +\mathcal{B}^\top *_M \mathcal{P} *_M \Lambda\right]\right)+ O(\varepsilon).\\
\end{aligned}
\end{equation*}
We used the cyclic property \ref{pro:cyclic_trace} and the transpose with trace property \ref{prop:trace_conjuguate}
 in the last equality. Then, as $\mathcal{B}$ is symmetric, the partial derivative in the direction $\mathcal{H}$ is
\begin{equation*}
\begin{aligned}
\dfrac{\partial f_2}{\partial \mathcal{P}}(\mathcal{H})&=\lim_{\varepsilon \rightarrow 0} \dfrac{f_2(\mathcal{P}+\varepsilon \mathcal{H})-f_2(\mathcal{P})}{\varepsilon}\\
&=\operatorname{Tr}\left(\mathcal{H}^\top *_M \mathcal{B} *_M \mathcal{P} *_M (\Lambda +\Lambda^\top )\right).
\end{aligned}
\end{equation*}
This yields the partial derivative of $\mathcal{L}$ with respect to $\mathcal{P}$ as follows
\begin{equation*}
\dfrac{\partial f_2}{\partial \mathcal{P}}=\mathcal{B} *_M \mathcal{P} *_M (\Lambda +\Lambda^\top ).
\end{equation*}
Subsequently,
\begin{equation*}
\begin{aligned}
\dfrac{\partial \mathcal{L}}{\partial \mathcal{P}}=0 & \iff 2 \mathcal{X} *_M \mathcal{P}- \Lambda *_M \mathcal{P} -\Lambda^\top *_M \mathcal{P}=0\\ 
& \iff \mathcal{X} *_M \mathcal{P}= \mathcal{B} *_M \mathcal{P} *_M (\Lambda +\Lambda^\top )\\ 
&\iff \mathcal{X} *_M \mathcal{P}= \mathcal{B} *_M \mathcal{P} *_M \widehat{\Lambda}\\
&\iff \mathcal{X} *_M \mathcal{P} = \mathcal{B} *_M \mathcal{P} *_M \mathcal{Q}^\top *_M \mathcal{D} *_M \mathcal{Q} \\
&\iff \mathcal{X} *_M \mathcal{P} *_M \mathcal{Q}^\top = \mathcal{B} *_M \mathcal{P} *_M \mathcal{Q}^\top *_M \mathcal{D}\\
&\iff \mathcal{X} *_M \widehat{\mathcal{P}}= \mathcal{B} *_M \widehat{\mathcal{P}} *_M \mathcal{D}.
\end{aligned}
\end{equation*}
The third line utilizes the property that $\widehat{\Lambda} = \Lambda + \Lambda^\top$, which is symmetric, thus diagonalizable (\ref{thm:Spectral}), i.e., $\widehat{\Lambda} = \mathcal{Q}^\top *_M \mathcal{D} *_M \mathcal{Q}$. The last two lines are justified by the fact that $\widehat{\mathcal{P}} = \mathcal{P} *_M \mathcal{Q}$ also satisfies $\widehat{\mathcal{P}}^\top *_M \mathcal{B} *_M \widehat{\mathcal{P}} = \mathcal{I}$, concluding the proof.
\end{proof}

\medskip
A corollary of this theorem can be deduced.
\medskip
\begin{corollary}\label{thm:Tr}
Let $\mathcal{X} \in \mathbb{R}^{I_1 \times \ldots \times I_M \times K_1 \times \ldots \times K_N}$, the solution of \[\arg \min_{\substack{\mathcal{P} \in \mathbb{R}^{I_1 \times \ldots \times I_M \times d}\\ \mathcal{P}^\top 
*_M \mathcal{P}=\mathcal{I}}} \|\mathcal{P}^\top *_M \mathcal{X}\|_F^2. \]
is the $d$ smallest left singular tensors of $\mathcal{X}$.
\end{corollary}

\begin{proof}
We have
$\|\mathcal{P}^\top *_M \mathcal{X}\|_F^2 =\operatorname{Tr} \left(\mathcal{P}^\top *_M \mathcal{X} *_N \mathcal{X}^\top *_M \mathcal{P}\right).$ Theorem~\ref{thm:tr_pos_def} tells that the solution is equivalent to solve $\mathcal{X} *_N \mathcal{X}^\top *_M \mathcal{P}=\lambda \mathcal{P} $, i.e., the $d$ smallest tensors of $\mathcal{X} *_N \mathcal{X}^\top$, which corresponds exactly to the $d$ smallest left singular tensors of $\mathcal{X}$.
\end{proof}

\section{Multidimensional reduction}\label{sec:DR:3D}
In this section, we present a generalized approach to DR methods using the Einstein product.

Given a tensor $\mathcal{X} \in \mathbb{R}^{I_1 \times \ldots \times I_M \times n}$, our objective is to derive a low-dimensional representation $\mathcal{Y} \in \mathbb{R}^{d \times n}$ of $\mathcal{X}$. This involves defining a mapping function $\Psi: \mathbb{R}^{I_1 \times \ldots \times I_M} \longrightarrow \mathbb{R}^{d}$.

First, we discuss the determination of the weight matrix, which can be computed in various ways, one common method is using the Gaussian kernel
\[ W_{i,j} = \exp\left(-\frac{\left\|\mathcal{X}^{(i)}-\mathcal{X}^{(j)}\right\|_{F}^{2}}{\sigma^{2}}\right). \]
Additionally, introducing a threshold can yield a sparse matrix (Gaussian-threshold). We also explore another method later in this section, which utilizes the reconstruction error.

Next, we introduce our proposed methods.\\

\textbf{Linear methods}: The linear methods can be written as $\mathcal{Y}=\mathcal{P}^\top *_M \mathcal{X}$. It is sufficient to find the projection matrix $\mathcal{P} \in \mathbb{R}^{I_1 \times \ldots \times I_M \times d}$.\\
Higher order PCA based on Einstein \cite{hachimi2023tensor} is the natural extension of PCA to higher order tensors. It extends the PCA applied to images, using the notion of eigenfaces to colored images that are modeled by a fourth-order tensor, using the Einstein product. It vectorizes pixels (height and width) for each color (RGB) to get a third-order tensor, then it computes E-SVD of this tensor centered, to get the eigenfaces.\\
The vectorization is not natural, since we omit the spatial information. The proposed work hides the vectorization in the first step by using the tensor directly, and seeks to find a solution of the following problem
\begin{equation}\label{eq:PCA_3D}
\begin{aligned}
\arg\max_{\substack{\mathcal{P} \in \mathbb{R}^{I_1 \times \ldots \times I_M \times d}\\ \mathcal{P}^\top 
*_M \mathcal{P}=\mathcal{I}\\ \mathcal{Y}=\mathcal{P}^\top *_M \mathcal{X}}} \Phi_{PCA}(\mathcal{Y})& :=\sum_{i}\left\|\mathcal{Y}^{(i)}-\dfrac{1}{n}\sum_{j} \mathcal{Y}^{(j)}\right\|_{F}^{2}.
\end{aligned}
\end{equation}
The objective function can be written as 
\begin{equation*}
\begin{aligned}
\Phi_{PCA}(\mathcal{Y})&=\sum_{i}\left\|\mathcal{P}^\top *_M (\mathcal{X}^{(i)}-\dfrac{1}{n}\sum_{j} \mathcal{X}^{(j)})\right\|_{F}^{2}\\
&=\sum_{i}\left\|\mathcal{P}^\top *_M (\mathcal{X}^{(i)}-\mathcal{Q}^{(i)})\right\|_{F}^{2}\\
&=\left\|\mathcal{P}^\top *_M (\mathcal{X}- \mathcal{Q})\right\|_{F}^{2},
\end{aligned}
\end{equation*}
with $ \mathcal{Q} \in \mathbb{R}^{I_1 \times \ldots \times I_M\times n} $, where $\mathcal{Q}^{(i)}=\frac{1}{n}\sum_{j} \mathcal{X}^{(j)}$ represents the mean.\\
The solution of \eqref{eq:PCA_3D} is the largest \(d\) left singular tensors of the centered tensor $ \mathcal{X}- \mathcal{Q}=\mathcal{X} \times_{M+1} H $.

Since the feature dimension is typically larger than the number of data points \( n \), computing the E-SVD of \( \mathcal{X} \times_{M+1} H \) can be computationally expensive. It's preferable to have a runtime that depends on \( n \) instead. To achieve this, we transform the equation $ \mathcal{X} \times_{M+1} H *_1 \mathcal{X}^\top *_M \mathcal{P}=\lambda \mathcal{P} $ to $ (\mathcal{X}^\top *_M \mathcal{X} \times_M H ) *_1 \mathbf{z}=\lambda \mathbf{z} $, with $ \mathbf{z}= \mathcal{X}^\top *_M \mathcal{P}$. This allows us to find the eigenvectors of a square matrix of size \( n \). The projected data \( \mathcal{Y} \) would be these vectors reshaped to the appropriate size.\\
The algorithm bellow shows the steps of PCA via the Einstein product.
\begin{algorithm}[H]
\caption{PCA-Einstein}
\label{algorithm_PCA}
\hspace*{\algorithmicindent} \textbf{Input:} $\mathcal{X}$ (Data) $d$(dimension output).\\
\hspace*{\algorithmicindent} \textbf{Output:} $\mathcal{P}$ (Projection space).
\begin{algorithmic}[1]
\State Compute $\mathcal{Q}$. \Comment{The mean tensor}
\State Compute the largest $d$ eigen-tensors of $\mathcal{Z} =\mathcal{X}-\mathcal{Q}$.
\State Combine these tensors to get $\mathcal{P}$.
\end{algorithmic}
\end{algorithm}

\subsection{Generalization of ONPP}
Given a weight matrix $W \in \mathbb{R}^{n \times n}$, the objective function is to resolve 
\begin{equation}\label{eq:ONPP_3D}
\begin{aligned}
\arg \min_{\substack{\mathcal{P} \in \mathbb{R}^{I_1 \times \ldots \times I_M \times d}\\ \mathcal{P}^\top 
*_M \mathcal{P}=\mathcal{I}\\ \mathcal{Y}=\mathcal{P}^\top *_M \mathcal{X}}} \Phi_{ONPP}(\mathcal{Y})& :=\sum_{i}\left\|\mathcal{Y}^{(i)}-\sum_{j} w_{i,j} \mathcal{Y}^{(j)}\right\|_{F}^{2}.
\end{aligned}
\end{equation}
The objective function can be written as
\begin{equation*}
\begin{aligned} 
\Phi_{ONPP}(\mathcal{Y})&=\sum_{i} \left\| \sum_j \delta_{i,j} \mathcal{Y}^{(j)} - \sum_j w_{i,j} \mathcal{Y}^{(j)} \right\|_{F}^{2}\\
&=\sum_{i} \left\| \sum_j (\delta_{i,j}- w_{i,j}) \mathcal{Y}^{(j)} \right\|_{F}^{2}\\
&=\sum_{i} \left\| \sum_j (I_n-W)_{i,j} \mathcal{Y}^{(j)} \right\|_{F}^{2}\\
&=\sum_{i} \left\| \left(\mathcal{Y} \times_{M+1} (I_n-W) \right)^{(i)} \right\|_{F}^{2}\\
&=\left\| \mathcal{Y} \times_{M+1} (I_n-W) \right\|_{F}^{2}\\
&=\left\|(\mathcal{P}^\top *_M \mathcal{X}) \times_{M+1} (I_n-W) \right\|_{F}^{2}\\
&=\left\|\mathcal{P}^\top *_M \left(\mathcal{X} \times_{M+1} \left(I_n-W \right)\right) \right\|_{F}^{2}.\\
\end{aligned}
\end{equation*}
Using corollary~\ref{thm:Tr}, the solution of \eqref{eq:ONPP_3D} is the smallest $d$ left singular tensors of $\mathcal{X} \times_{M+1} (I_n-W)$.\\
The algorithm bellow, shows the steps of ONPP via the Einstein product.
\begin{algorithm}[H]
\caption{ONPP-Einstein}
\label{algorithm_ONPP}
\hspace*{\algorithmicindent} \textbf{Input:} $\mathcal{X}$ (Data) $d$ (subspace dimension).\\
\hspace*{\algorithmicindent} \textbf{Output:} $\mathcal{P}$ (Projection space).
\begin{algorithmic}[1]
\State Compute $W$. \Comment{Using the appropriate method} 
\State Compute the smallest $d$ left singular tensors of $\mathcal{Z} =\mathcal{X} \times_{M+1} (I-W)$.
\State Combine these tensors to get $\mathcal{P}$.
\end{algorithmic}
\end{algorithm}

\subsubsection{Multi-weight ONPP}
In this section, we propose a generalization of ONPP, where multiple weights matrices are employed on the \(I_M\) mode. We denote by \( \mathcal{W} \in \mathbb{R}^{n \times n \times I_M} \) the weight tensor. Let \( \mathcal{Y}^{(i)}_r \) denotes the tensor $\mathcal{Y}$, by fixing its last two indices to $(r,i)$, i.e., \( \mathcal{Y}_{:,\ldots,r,i} \), similarly \( \mathcal{X}^{(i)}_r \) denotes \( \mathcal{X}_{:,\ldots,r,i} \). We assume that the \(r\)-th frontal slice of the weight tensor is constructed only from the \(r\)-th frontal slice of the data tensor.\\
The objective function is
\begin{equation}
\begin{aligned}
\arg \min_{\substack{\mathcal{P} \in \mathbb{R}^{I_1 \times \ldots \times I_M \times d}\\ \mathcal{P}^\top 
*_M \mathcal{P}=\mathcal{I} \\ \mathcal{Y}=\mathcal{P}^\top *_M \mathcal{X}}} \Phi_{ONPP_{MW}}(\mathcal{Y})& :=\sum_{i,r} \left\|\mathcal{Y}^{(i)}_r-\sum_{j} \mathcal{W}^{(r)}_{i,j} \mathcal{Y}^{(j)}_r\right\|_{F}^{2}.
\end{aligned}
\end{equation}

For each \(i\), utilizing the independence of the frontal slices \( \mathcal{Y}^{(i)}_r \), we can divide the objective function into \(I_M\) independent objective functions. The solution is obtained by concatenating the solutions of each objective function. The \(r\)-th objective function can be written as
\begin{equation}
\begin{aligned}
\arg \min_{\substack{\mathcal{P}_{:,\ldots,r,:} \in \mathbb{R}^{I_1 \times \ldots \times I_M \times d}\\\mathcal{P}_{:,\ldots,r,:}^\top *_M \mathcal{P}_{:,\ldots,r,:}=\mathcal{I} \\ \mathcal{Y}{:,\ldots,r,:}=\mathcal{P}{:,\ldots,r,:}^\top *_M \mathcal{X}{:,\ldots,r,:}}} & \sum_{i} \left\|\mathcal{Y}^{(i)}_r-\sum_{j} \mathcal{W}^{(r)}_{i,j} \mathcal{Y}^{(j)}_r\right\|_{F}^{2}.
\end{aligned}
\end{equation}
The solution of this objective function is the smallest \(d\) left singular tensors of\\ $\mathcal{X}_{:,\ldots,r,:} \times_{M+1} (I_n-\mathcal{W}^{(r)}) $. The solution of the original problem is obtained by concatenating the solutions of each objective function.

\subsection{Generalization of OLPP}
Given a weight matrix $W \in \mathbb{R}^{n \times n}$, the optimization problem to solve is
\begin{equation}
\begin{aligned}
\arg \min_{\substack{\mathcal{P} \in \mathbb{R}^{I_1 \times \ldots \times I_M \times d}\\ \mathcal{P}^\top 
*_M \mathcal{P}=\mathcal{I}}} \Phi_{OLPP}(\mathcal{Y})& :=\dfrac{1}{2} \sum_{i,j} w_{i,j} \left\|\mathcal{Y}^{(i)}-\mathcal{Y}^{(j)}\right\|_{F}^{2}.
\end{aligned}
\end{equation}
The objective function can be written as
\begin{equation*}
\begin{aligned}
\Phi_{OLPP}(\mathcal{Y})&= \dfrac{1}{2} \sum_{i,j} w_{i,j} \left\|\mathcal{Y}^{(i)}\right\|_{F}^{2} + w_{i,j} \left\|\mathcal{Y}^{(j)}\right\|_{F}^{2} - \langle \mathcal{Y}^{(i)}, \mathcal{Y}^{(j)}\rangle \\
&= \dfrac{1}{2} \sum_{i} d_i \left\|\mathcal{Y}^{(i)}\right\|_{F}^{2} + \dfrac{1}{2} \sum_{j} d_j \left\|\mathcal{Y}^{(j)}\right\|_{F}^{2} - \sum_{i,j} w_{i,j} \langle \mathcal{Y}^{(i)}, \mathcal{Y}^{(j)}\rangle \\
&=\sum_{i,j} d_i \left\|\mathcal{Y}^{(i)}\right\|_{F}^{2}- \sum_{i,j} w_{i,j} \langle \mathcal{Y}^{(i)}, \mathcal{Y}^{(j)}\rangle \\
&=\sum_{i,j} d_{i,j} \langle \mathcal{Y}^{(i)}, \mathcal{Y}^{(j)}\rangle- \sum_{i,j} w_{i,j} \langle \mathcal{Y}^{(i)}, \mathcal{Y}^{(j)}\rangle \\
&=\sum_{i,j} L_{i,j} \langle \mathcal{Y}^{(i)}, \mathcal{Y}^{(j)}\rangle\\
&= \langle \mathcal{Y} *_{M+1} L, \mathcal{Y}\rangle\\
&=\operatorname{Tr}\left(\mathcal{P}^\top *_M (\mathcal{X} \times_{M+1} L *_1 \mathcal{X}^\top ) *_M \mathcal{P}\right),
\end{aligned}
\end{equation*}
where $L$ is the Laplacian matrix corresponding to $W$.\\
The solution using Theorem ~\ref{thm:Tr}, is the smallest $d$ eigen-tensors of the symmetric tensor (Prop. \ref{prop:symmetric_XMXt}) $\mathcal{X} \times_{M+1} L *_1 \mathcal{X}^\top$.\\
The algorithm bellow, shows the steps of OLPP via the Einstein product.
\begin{algorithm}[H]
\caption{OLPP-Einstein}
\label{algorithm_OLPP}
\hspace*{\algorithmicindent} \textbf{Input:} $\mathcal{X}$ (Data) $d$ (subspace dimension).\\
\hspace*{\algorithmicindent} \textbf{Output:} $\mathcal{P}$ (Projection space).
\begin{algorithmic}[1]
\State Compute $L$. \Comment{Using the appropriate method} 
\State Compute the smallest $d$ eigen-tensors of $\mathcal{X} \times_{M+1} L *_1 \mathcal{X}^\top$.
\State Combine these tensors to get $\mathcal{P}$.
\end{algorithmic}
\end{algorithm}

\subsubsection{Multi-weight OLPP}
In this section, we propose a generalization of OLPP, where multiple weights are utilized on the \(I_M\) mode. The objective function is to solve
\begin{equation}
\begin{aligned}
\arg \min_{\substack{\mathcal{P} \in \mathbb{R}^{I_1 \times \ldots \times I_M \times d}\\ \mathcal{P}^\top 
*_M \mathcal{P}=\mathcal{I} \\ \mathcal{Y}=\mathcal{P}^\top *_M \mathcal{X}}} \Phi_{OLPP_{MW}}(\mathcal{Y})& :=\dfrac{1}{2} \sum_{i,j,r} \mathcal{W}^{(r)}_{i,j} \left\|\mathcal{Y}_r^{(i)}-\mathcal{Y}_r^{(j)}\right\|_{F}^{2}.
\end{aligned}
\end{equation}

Here, we have \(I_M\) independent objective functions, and the solution is obtained by concatenating the solutions of each objective function.

\subsection{Generalization of LPP}
LPP is akin to the Laplacian Eigenmap, serving as its linear counterpart. The objective function of LPP solves 
\begin{equation}
\begin{aligned}
\arg \min_{\substack{\mathcal{P} \in \mathbb{R}^{I_1 \times \ldots \times I_M \times d}\\ \mathcal{P}^\top *_M \mathcal{X} \times_{M+1} D *_1 \mathcal{X}^\top *_M \mathcal{P} =\mathcal{I}}} \Phi_{LPP}(\mathcal{Y}) & :=\dfrac{1}{2} \sum_{i,j} w_{i,j} \left\|\mathcal{Y}^{(i)}-\mathcal{Y}^{(j)}\right\|_{F}^{2}.
\end{aligned}
\end{equation}
The solution involves finding the smallest $d$ eigen-tensor of the generalized eigen-problem
\begin{equation}\label{eq:LPP_sol}
\mathcal{X} \times_{M+1} L *_1 \mathcal{X}^\top *_M \mathcal{V}=\lambda \mathcal{X} *_{M+1} D *_1 \mathcal{X} ^\top *_M \mathcal{V}.
\end{equation}
Using \ref{prop:symmetric_XMXt}, we deduce that the tensors $\mathcal{X} \times_{M+1} L *_1 \mathcal{X}^\top *_M$, $\mathcal{X} \times_{M+1} D *_1 \mathcal{X}^\top *_M$, are symmetric positive semi-definite. Here, we assume definiteness (although generally not true, especially if the number of sample points is less than the product of the feature dimensions). The projection tensor is obtained by adding these eigen-tensors and reshaping it to the desired dimension.\\
The algorithm bellow, shows the steps of LPP via the Einstein product.
\begin{algorithm}[H]
\caption{LPP-Einstein}
\label{algorithm_LPP}
\hspace*{\algorithmicindent} \textbf{Input:} $\mathcal{X}$ (Data) $d$ (subspace dimension).\\
\hspace*{\algorithmicindent} \textbf{Output:} $\mathcal{P}$ (Projection space).
\begin{algorithmic}[1]
\State Compute $L$. \Comment{Using the appropriate method} 
\State Compute the smallest $d$ eigen-tensors of $\mathcal{X} \times_{M+1} L *_1 \mathcal{X}^\top *_M \mathcal{V}=\lambda \mathcal{X}^\top *_{M+1} D *_1 \mathcal{X} ^\top *_M \mathcal{V}$. 
\State Combine these tensors to get $\mathcal{P}$.
\end{algorithmic}
\end{algorithm}

Similarly, the multi-weight LPP can be proposed.

\subsection{Generalization of NPP}
The generalization of NPP resembles that of ONPP, under the constraint of LLE, it aims to find
\begin{equation}
\begin{aligned}
\arg \min_{\substack{\mathcal{P} \in \mathbb{R}^{I_1 \times \ldots \times I_M \times d}\\ \mathcal{P}^\top *_M \mathcal{X} *_1 \mathcal{X}^\top *_M \mathcal{P}=\mathcal{I}}} \Phi_{NPP}(\mathcal{Y}) &:= \operatorname{Tr}\left(\mathcal{P}^\top *_M \mathcal{X} \times_{M+1} (I_n-W)(I_n-W)^\top \times_1 \mathcal{X}^\top *_M \mathcal{P}\right).
\end{aligned}
\end{equation}
The solution entails finding the smallest \( d \) eigenvectors of the generalized eigen-problem
\begin{equation*}
\mathcal{X} \times_{M+1} (I_n-W)(I_n-W)^\top \times_1 \mathcal{X}^\top *_M \mathcal{V} =\lambda \mathcal{X} *_1 \mathcal{X}^\top *_M \mathcal{V}.
\end{equation*}
The projection tensor is obtained by concatenating these eigenvectors into a matrix and then reshaping it to the desired dimension.\\
The algorithm bellow, shows the steps of NPP via the Einstein product.
\begin{algorithm}[H]\caption{NPP-Einstein}
\label{algorithm_NPP}
\hspace*{\algorithmicindent} \textbf{Input:} $\mathcal{X}$ (Data) $d$ (subspace dimension).\\
\hspace*{\algorithmicindent} \textbf{Output:} $\mathcal{P}$ (Projection space).
\begin{algorithmic}[1]
\State Compute $W$. \Comment{Using the appropriate method} 
\State Compute the smallest $d$ eigenvectors of $\mathcal{X} \times_{M+1} (I_n-W)(I_n-W)^\top \times_1 \mathcal{X}^\top *_M \mathcal{V} =\lambda \mathcal{X} *_1 \mathcal{X}^\top *_M \mathcal{V}$.
\State Combine these tensors to get $\mathcal{P}$.
\end{algorithmic}
\end{algorithm}
\noindent
Similarly, the multi-weight NPP can be proposed, and the solution is similar.\\

\textbf{Nonlinear methods}: Nonlinear DR methods are potent tools for uncovering the nonlinear structure within data. However, they present their own challenges, such as the absence of an inverse mapping, which is essential for data reconstruction. Another difficulty encountered is the out-of-sample extension, which involves extending the method to new data. While a variant of these methods utilizing multiple weights could be proposed, it would resemble the approach of linear methods, and thus we will not delve into them here.
\subsection{Generalization of Laplacian Eigenmap}
Given a weight matrix $W \in \mathbb{R}^{n \times n}$, the objective function is to solve
\begin{equation}
\begin{aligned}
\arg \min_{\mathcal{Y} \times_{M+1} D *_1 \mathcal{Y}^\top =\mathcal{I}} \Phi_{LE}(\mathcal{Y}) & :=\dfrac{1}{2} \sum_{i,j} w_{i,j} \left\|\mathcal{Y}^{(i)}-\mathcal{Y}^{(j)}\right\|_{F}^{2}.
\end{aligned}
\end{equation}
The objective function can be written as
\begin{equation*}
\Phi_{LE}(\mathcal{Y})=\operatorname{Tr}\left(\mathcal{Y} \times_{M+1} L *_1 \mathcal{Y}^\top \right).
\end{equation*}
For $\widehat{\mathcal{Y}}=\mathcal{Y} \times_{M+1} D^{1 / 2}$, the constraint becomes 
$\widehat{\mathcal{Y}} *_1 \widehat{\mathcal{Y}}^\top =\mathcal{I}$, using $\mathcal{Y}=\widehat{\mathcal{Y}} \times_{M+1} D^{-1 / 2}$, the objective function becomes
\begin{equation} \label{eq:OLPP_equivalent}
\Phi_{LE}(\widehat{\mathcal{Y}})=\operatorname{Tr}\left(\widehat{\mathcal{Y}} \times_{M+1} L_n *_1 \widehat{\mathcal{Y}}^\top \right).
\end{equation}
Using the isomorphism $\Psi$ and its properties, the problem is equivalent to Equation~\eqref{eq : loss1} under the constraint $\Psi(\widehat{\mathcal{Y}}) *_1 \Psi(\widehat{\mathcal{Y}})^\top=I$. Since the solution is the smallest $d$ eigenvectors of $L_n$, the solution of the original problem would be $\mathcal{Y}=\Psi^{-1}(\Psi(\widehat{\mathcal{Y}})) \times_{M+1} D^{-1 / 2}$.\\
The algorithm bellow, shows the steps of LE via the Einstein product.
\begin{algorithm}[H]
\caption{LE-Einstein}
\label{algorithm_LE}
\hspace*{\algorithmicindent} \textbf{Input:} $\mathcal{X}$ (Data) $d$ (subspace dimension).\\
\hspace*{\algorithmicindent} \textbf{Output:} $\mathcal{Y}$ (Projection data).
\begin{algorithmic}[1]
\State Compute $L_n$. \Comment{Using the appropriate method} 
\State Compute the smallest $d$ vectors of $L_n$.
\State Combine these vectors and reshape them to get $\Psi(\widehat{\mathcal{Y}})$.
\State Compute $\mathcal{Y}=\Psi^{-1}(\Psi(\widehat{\mathcal{Y}})) \times_{M+1} D^{-1 / 2}$.
\end{algorithmic}
\end{algorithm}

\subsubsection{Projection on out of Sample Data}
The out of sample extension is the problem of extending the method to new data. Many approaches were proposed to solve this problem, such as the Nyström method, kernel mapping, eigenfunctions \cite{bengio2003out} \cite{bengio2003out,tai2022kernelized}, etc. We will propose a method that is based on the eigenfunctions.\\
In matrix case, the out of sample extension is done simply computing the components explicitly as $y_{t_j}=\frac{1}{\lambda_j} \mathbf{k}_{t}^\top \mathbf{v}_j, \; j=1,\ldots, d$, where $\mathbf{k}_{t}$ is the kernel matrix of the new data, and $(\lambda_j,\mathbf{v}_j)$ is eigentuple of the kernel matrix $L_n$, $\mathbf{k}_t=(K(\mathbf{x}_t,\mathbf{x}_1),\ldots,K(\mathbf{x}_t,\mathbf{x}_n))^\top$ is the kernel vector of the test data $\mathbf{x}_t$. 
It can be written as 
$\mathbf{y}_{t}= \operatorname{diag}(\lambda_1,\ldots,\lambda_d)^{-1} [\mathbf{v}_1,\ldots,\mathbf{v}_d]^\top \mathbf{k}_{t} $
Thus, the generalization is straightforward.

\subsection{Generalization of LLE}
LLE is a nonlinear method that aims to preserve the local structure of the data. After finding the \( k \)-nearest neighbors, it determines the weights that minimize the reconstruction error. In other words, it seeks to solve the following objective function
\begin{equation}\label{eq:LLE_weight}
Re(W):=\sum_{i}\left\|\mathcal{X}^{(i)}-\sum_{j} w_{i,j} \mathcal{X}^{(j)}\right\|_{F}^{2},
\end{equation}
subject to two constraints: a sparse one (the weights are zero if a point is not in the neighborhood of another point), ensuring the locality of the reconstruction weight, and an invariance constraint (the row sum of the weight matrix is 1). Finally, the projected data is obtained by solving the following objective function
\begin{equation}
\begin{aligned}
\arg \min_{\substack{\mathcal{Y} *_1 \mathcal{Y}^\top =\mathcal{I}\\ \sum_{i} \mathcal{Y}^{(i)} = \mathcal{O}}} \Phi_{LLE} & :=\sum_{i}\left\|\mathcal{Y}^{(i)}-\sum_{j} w_{i,j} \mathcal{Y}^{(j)}\right\|_{F}^{2}=\left\|\mathcal{Y} \times_{M+1} (I_n-W) \right\|_{F}^{2}.
\end{aligned}
\end{equation}

\subsubsection{Computing the weights}
Equation~\eqref{eq:LLE_weight} can be decomposed to finding the weights $w_{i,:}$ of a point $\mathcal{X}^{(i)}$ independently. For simplicity, we will refer to the neighbors of $\mathcal{X}^{(i)}$ as $\mathcal{N}^{(i,j)}$, and the weights of its neighbors as $\mathbf{w}_i \in \mathbb{R}^{k}$, with $k$ the number of neighbors, which plays the rule of the sparseness constraint. Denote the local Grammian matrix $G^{(i)} \in \mathbb{R}^{k \times k}$ as $G^{(i)}_{j,k}= \langle \mathcal{X}^{(i)}-\mathcal{N}^{(i,j)},\mathcal{X}^{(i)}-\mathcal{N}^{(i,k)} \rangle$.\\
The invariance constraint can be written as $\mathbf{1}^\top \mathbf{w}_{i}=1$, thus, we can write the problem as
\begin{equation*}
\begin{aligned}
\sum_i \left\|\mathcal{X}^{(i)}-\sum_{j} w_{i,j} \mathcal{X}^{(j)}\right\|_{F}^{2}&= \sum_i \left\|\mathcal{X}^{(i)}-\sum_{j} w_{j} \mathcal{N}^{(i,j)}\right\|_{F}^{2}\\
&= \sum_i \left\|\sum_{j} w_{i,j}(\mathcal{X}^{(i)} -\mathcal{N}^{(i,j)})\right\|_{F}^{2}\\
&=\sum_{i,j,k} w_{i,j} w_{i,k} G^{(i)}_{j,k}\\
&=\sum_i \mathbf{w}_i^\top G^{(i)} \mathbf{w}_i.
\end{aligned}
\end{equation*}
This constrained problem can be solved using the Lagrangian
\begin{equation*}
\mathcal{L}(\{\mathbf{w}_i\}_i,\lambda)= \sum_i \mathbf{w}_i^\top G^{(i)} \mathbf{w}_i- \sum_i \lambda_i (\mathbf{1}^\top \mathbf{w}_i-1).
\end{equation*}
We compute the partial derivative with respect to $\mathbf{w}_i$ and setting it to zero
\begin{equation}\begin{aligned}\label{eq:LLE_weight_KKT_1}
\dfrac{\partial \mathcal{L}}{\partial \mathbf{w}_i}&=2G^{(i)} \mathbf{w}_i-\lambda_i \mathbf{1}=0\\
\implies \mathbf{w}_i&=\frac{\lambda_i}{2} G^{(i)^{-1}} \mathbf{1}.\\
\end{aligned}
\end{equation}
We utilize the fact that the Grammian matrix is symmetric, assuming that \( G^{(i)} \) is full rank, which is typically the case. However, if \( G^{(i)} \) is not full rank, a small value can be added to its diagonal to ensure full rank.\\
The partial derivative with respect to \( \lambda_i \), set to zero, yields the invariance constraint of point \( i \), i.e., \( \mathbf{1}^\top \mathbf{w}_i=1 \). Multiplying Equation \eqref{eq:LLE_weight_KKT_1} by \( \mathbf{1}^\top \), we can isolate \( \lambda \) and arrive at the following equation
\begin{equation*}
\mathbf{w}_i=\dfrac{G^{(i)^{-1}} \mathbf{1}}{\mathbf{1}^\top G^{(i)^{-1}} \mathbf{1}}.
\end{equation*}

\subsubsection{Computing the projected data}
The final step resembles the previous cases, ensuring that the mean constraint is satisfied. As $I-W$ can be represented as a Laplacian, we know that the number of components corresponds to the multiplicity of the eigenvalue 0. Hence, there is at least one eigenvalue 0 with multiplicity 1, and the identity tensor serves as the corresponding eigenvector, thereby satisfying the second constraint. For further elaboration, interested readers can refer to \cite{ghojogh2020locally}.\\
The solution is equivalent to solving the matricized version, where the solution comprises the smallest $d$ singular eigenvectors of the matrix \( (I_n-W)(I_n-W^\top ) \). Consequently, the solution to the original problem is the inverse transform, denoted as \( \Psi^{-1} \), of the solution to the matricized version.\\
The algorithm bellow, shows the steps of LLE via the Einstein product.
\begin{algorithm}[H]
\caption{LLE-Einstein}
\label{algorithm_LLE}
\hspace*{\algorithmicindent} \textbf{Input:} $\mathcal{X}$ (Data) $d$ (subspace dimension).\\
\hspace*{\algorithmicindent} \textbf{Output:} $\mathcal{Y}$ (Projection data).
\begin{algorithmic}[1]
\State Find the neighbors of each point. 
\State Compute the reconstruction weight $W$.
\State Compute the smallest $d$ eigenvectors of $(I_n-W)(I_n-W^\top )$.
\State Compute the $\Psi^{-1}$ of these vectors with the appropriate reshaping to get $\mathcal{Y}$.
\end{algorithmic}
\end{algorithm}

\subsubsection{Projection on out of Sample Data}
To extend these methods to new data (test data) not seen in the training set, various approaches can be employed. These include kernel mapping, eigenfunctions (as discussed in Bengio et al. \cite{bengio2003out}), and linear reconstruction. Here, we opt to generalize the latter approach. We can follow similar steps to those used in matrix-based methods. Specifically, we can perform the following steps without re-running the algorithm on the entire dataset
\begin{enumerate}
 \item Find the neighbors in training data of each new data test.
 \item Compute the reconstruction weight that best reconstruct each test point from its k neighbours in the training data.
 \item Compute the low dimensional representation of the new data using the reconstruction weight.
\end{enumerate}
More formally, after finding the neighbours $\mathcal{N}^{(i,j)}$ of a test data $\mathcal{X}_t^{(i)}$, we solve the following problem
\begin{equation*}
\arg \min_{w^{(t)}}\sum_i \left\|\mathcal{X}_t^{(i)}-\sum_{j} w^{(t)}_{i,j} \mathcal{N}^{(i,j)}\right\|_{F}^{2},
\end{equation*}
with $\mathbf{w}^{(t)}_{i,:}$ is the reconstruction weight of the test data, under the same constraint, i.e., the row sum of the weight matrix is 1, with value zero if it's not in the neighbour of a point.\\
The solution is $w^{(t)}_{i}=\dfrac{G_t^{(i)^{-1}} \mathbf{1}}{\mathbf{1}^\top G_t^{(i)^{-1}} \mathbf{1}}$, with $G^{(i)}_{t_{j,k}}= \langle \mathcal{X}_t^{(i)}-\mathcal{N}^{(i,j)},\mathcal{X}_t^{(i)}-\mathcal{N}^{(i,k)} \rangle$.\\
Finally, the embedding of the test data $\mathcal{Y}_t^{(i)}$ is obtained as $\sum_j w^{(t)}_{i,j} \mathcal{Y}^{(j)}$, with $\mathcal{Y}^{(j)}=\phi_{LLE}(\mathcal{N}^{(i,j)})$ are the embedding representation of the neighbours of the test data.

\section{Other variants via the Einstein product}\label{sec:Variants}
\subsection{Kernel methods}
Kernels serve as powerful tools enabling linear methods to operate effectively in high-dimensional spaces, allowing for the representation of nonlinear data without explicitly computing data coordinates in this feature space. The kernel trick, a pivotal breakthrough in machine learning, facilitates this process. Formally, instead of directly manipulating the data \(\mathcal{X}=\{\mathcal{X}^{(1)},\ldots,\mathcal{X}^{(n)}\}\), we operate within a high-dimensional implicit feature space using a function \(\Phi\). \\
We denote the tensor \(\Phi(\mathcal{X})=\{\Phi(\mathcal{X}^{(1)}),\ldots,\Phi(\mathcal{X}^{(n)})\}=\{\Phi(\mathcal{X})^{(1)},\ldots,\Phi(\mathcal{X})^{(n)}\}\). The mapping need not be explicitly known; only the Gram matrix \(K\) is required. This matrix represents the inner products of the data in the feature space, defined as \(K_{i,j}=\langle \Phi(\mathcal{X}^{(i)}),\Phi(\mathcal{X}^{(j)})\rangle\). Consequently, any method expressible in terms of data inner products can be reformulated in terms of the Gram matrix. Consequently, the kernel trick can be applied sequentially. Moreover, extending kernel methods to multi-linear operations using the Einstein product is straightforward. Commonly used kernels include the Gaussian, polynomial, Laplacian, and Sigmoid kernels, among others.\\
Denote \(\mathcal{Y}=\mathcal{P}^\top *_M \Phi(\mathcal{X})\), with \(\mathcal{P} \in \mathbb{R}^{I_1 \times \ldots \times I_M \times d}\).\subsubsection{Kernel PCA via the Einstein product}
The kernel multi-linear PCA solves the following problem
\begin{equation}
\begin{aligned}
\arg \max_{\substack{\mathcal{P} \in \mathbb{R}^{I_1 \times \ldots \times I_M \times d} \\ \mathcal{P}^\top *_M \mathcal{P}=\mathcal{I}}} &
\left\| \mathcal{P}^\top *_1 ( \Phi(\mathcal{X})- \mathcal{Q})\right\|_{F}^{2},\\
\end{aligned}
\end{equation}
with $\mathcal{Q}$ is the mean of kernel points, the solution is the largest $d$ left singular tensors of $\Phi(\mathcal{X})- \mathcal{Q}=\widehat{\Phi}(\mathcal{X})$. It needs to calculate the eigen-tensors of $\widehat{\Phi}(\mathcal{X}) *_1 \widehat{\Phi}(\mathcal{X})^\top$, which is not accessible, however, the Grammian $K= \widehat{\Phi}(\mathcal{X})^\top *_M \widehat{\Phi}(\mathcal{X})$ is available, we can transform the problem to
\begin{equation*}
\widehat{K} \widehat{\mathbf{z}}_{i}=\lambda_i \widehat{\mathbf{z}}_{i} \; \text{ with } \widehat{\mathbf{z}}_{i}=(\widehat{\Phi}(\mathcal{X}))^\top *_M \mathcal{Z}^{(i)} \in \mathbb{R}^{n},
\end{equation*}
with $\widehat{K}$ representing the Grammian of the centered data, that can easily be obtained from only $K$ as $\widehat{K}=K-HK- KH+HKH$, since 
\begin{equation*}
\begin{aligned}
\widehat{K}_{i,j}&=\widehat{\Phi}\left(\mathcal{X}\right)^{(i)^\top} *_1 \widehat{\Phi}(\mathcal{X})^{(j)}\\
&=\left(\Phi(\mathcal{X}^{(i)})^\top- \dfrac{1}{n} \sum_k \Phi(\mathcal{X}^{(k)})^\top\right) *_1\left(\Phi(\mathcal{X}^{(j)})- \dfrac{1}{n} \sum_l \Phi(\mathcal{X}^{(l)})^\top\right)\\
&=\Phi(\mathcal{X}^{(i)})^\top *_1 \Phi(\mathcal{X}^{(j)})- \dfrac{1}{n} \sum_k \Phi(\mathcal{X}^{(i)})^\top *_1 \Phi(\mathcal{X}^{(k)})- \dfrac{1}{n} \sum_l \Phi(\mathcal{X}^{(l)})^\top *_1 \Phi(\mathcal{X}^{(j)})\\
&=K_{i,j}- \dfrac{1}{n} \sum_k K_{i,k}1_{k,j}- \dfrac{1}{n} \sum_l 1_{i,l} K_{l,j}+ \dfrac{1}{n^2} \sum_{k,l} K_{k,l}\\
&=(K-\dfrac{\mathbf{1}}{n}K- K\dfrac{\mathbf{1}}{n}+ \dfrac{\mathbf{1}}{n}K\dfrac{\mathbf{1}}{n})_{i,j}.\\
\end{aligned}
\end{equation*}
Then, the solution is the same as the matrix case, i.e., the largest $d$ eigenvectors of $\widehat{K}$, reshaped to the appropriate size.\\
The algorithm bellow, shows the steps of the kernel PCA via the Einstein product.

\begin{algorithm}[H]
\caption{Kernel PCA-Einstein}
\label{algorithm_K_PCA}
\hspace*{\algorithmicindent} \textbf{Input:} $\mathcal{X}$ (Data) $(d_i)_{i \leq M}$(dimension output) $K$ (Grammian).\\
\hspace*{\algorithmicindent} \textbf{Output:} $\mathcal{Y}$ (Projected data).
\begin{algorithmic}[1]
\State Compute $\widehat{K}$. \Comment{The mean of the Grammian}
\State Compute the largest $d$ eigenvectors of $\widehat{K}$.
\State Combine these vectors, and reshape them to get $\mathcal{Y}$.
\end{algorithmic}
\end{algorithm}

\subsubsection{Kernel LPP via the Einstein product}
The kernel multi-linear LPP tackles the following problem
\begin{equation}
\begin{aligned}
\arg \min_{\substack{\mathcal{P} \in \mathbb{R}^{I_1 \times \ldots \times I_M \times d} \\ \mathcal{P}^\top *_M \Phi(\mathcal{X}) \times_{M+1} D *_1 \Phi(\mathcal{X})^\top *_M \mathcal{P} =\mathcal{I}}} &
 \operatorname{Tr} \left(\mathcal{P}^\top *_M \left( \Phi(\mathcal{X}) \times_{M+1} L *_1 \Phi(\mathcal{X})^\top \right) *_M \mathcal{P}\right).\\
\end{aligned}
\end{equation}
The solution involves finding the smallest $d$ eigen-tensors of the generalized eigen-problem
\begin{equation*}
\Phi(\mathcal{X}) \times_{M+1} L *_1 \Phi(\mathcal{X})^\top *_M \mathcal{V}=\lambda \Phi(\mathcal{X}) *_{M+1} D *_1 \Phi(\mathcal{X})^\top *_M \mathcal{V}.
\end{equation*}
$\Phi(\mathcal{X})$ is not available, the problem needs to be reformulated, Utilizing the fact that $K$ is invertible, we reformulate the problem as to find the vectors $\mathbf{z}$, solution of the generalized eigen-problem
\begin{equation*}
L \mathbf{z}=\lambda D \mathbf{z}, \text{ with } \mathbf{z}=\Phi(\mathcal{X})^\top *_M \mathcal{V}^{(i)} \in \mathbb{R}^{n},
\end{equation*}
This formulation reduces to the same minimization problem as in the matrix case.

\subsubsection{Kernel ONPP via the Einstein product}
The kernel multi-linear ONPP addresses the following optimization problem
\begin{equation}
\begin{aligned}
\arg \min_{\substack{\mathcal{P} \in \mathbb{R}^{I_1 \times \ldots \times I_M \times d} \\ \mathcal{P}^\top *_M \mathcal{P} =\mathcal{I}}} &
\operatorname{Tr} \left(\mathcal{P}^\top *_M (\Phi(\mathcal{X}) \times_{M+1} (I-W)(I-W^\top) *_1 \Phi(\mathcal{X})^\top) *_M \mathcal{P}\right).\\
\end{aligned}
\end{equation}
The solution involves finding the smallest $d$ eigen-tensors of problem
\begin{equation*}
\Phi(\mathcal{X}) \times_{M+1} (I-W)(I-W^\top) *_1 \Phi(\mathcal{X})^\top *_M \mathcal{V}=\lambda \mathcal{V}.
\end{equation*}
By employing similar techniques as before, we can derive the equivalent problem that seeks to find the vectors $\mathbf{z}$, solution of of the eigen-problem
\begin{equation*}
K(I-W)(I-W^\top) \mathbf{z} = \lambda \mathbf{z}, \; \text{ with } \Phi(\mathcal{X})^\top *_M \mathcal{V}=\mathbf{z}.
\end{equation*} 
The problem is the same minimization problem, the solution $\mathcal{Y}$ can be obtained from reshaping the transpose of the concatenated vectors $\mathbf{z}$.\\
The algorithm bellow, shows the steps of the kernel ONPP via the Einstein product.

\begin{algorithm}[H]
\caption{Kernel ONPP-Einstein}
\label{algorithm_K_ONPP}
\hspace*{\algorithmicindent} \textbf{Input:} $K$ (Grammian) $d$ (subspace dimension).\\
\hspace*{\algorithmicindent} \textbf{Output:} $\mathcal{Y}$ (Projected data).
\begin{algorithmic}[1]
\State Compute $W$. \Comment{Using the appropriate method} 
\State Compute the smallest $d$ eigenvectors of $K(I-W)(I-W^\top)$.
\State Combine these vectors, and reshape them to get $\mathcal{Y}$.
\end{algorithmic}
\end{algorithm}

\subsubsection{Kernel OLPP via the Einstein product}
The kernel multi-linear OLPP tackles the optimization problem defined as follows:
\begin{equation}
\begin{aligned}
\arg \min_{\substack{\mathcal{P} \in \mathbb{R}^{I_1 \times \ldots \times I_M \times d} \\ \mathcal{P}^\top *_M \mathcal{P} =\mathcal{I}}} &
\operatorname{Tr}\left(\mathcal{P}^\top *_M ( \Phi(\mathcal{X}) \times_{M+1} L *_1 \Phi(\mathcal{X})^\top ) *_M \mathcal{P}\right).\\
\end{aligned}
\end{equation}
The solution of the problem involves the eigen-tensors of 
$ \Phi(\mathcal{X}) \times_{M+1} L *_1 \Phi(\mathcal{X})^\top *_M \mathcal{Z}=\lambda \mathcal{Z}$,
By transforming the problem, we arrive to find the vectors $\mathbf{z}$, solution of of the eigen-problem
\begin{equation*}
K \mathbf{z} = \lambda \mathbf{z}, \; \text{ with } \Phi(\mathcal{X})^\top *_M \mathcal{Z}=\mathbf{z},
\end{equation*}
which mirrors the matrix case. Here, the solution $\mathcal{Y}$ can be obtained from reshaping the transpose of the concatenated vectors $\mathbf{z}$.\\
The algorithm bellow, shows the steps of the kernel OLPP via the Einstein product.

\begin{algorithm}[H]
\caption{Kernel OLPP-Einstein}
\label{algorithm_K_OLPP}
\hspace*{\algorithmicindent} \textbf{Input:} $K$ (Grammian) $d$ (subspace dimension).\\
\hspace*{\algorithmicindent} \textbf{Output:} $\mathcal{Y}$ (Projected data).
\begin{algorithmic}[1]
\State Compute the smallest $d$ eigenvectors of $K$.
\State Combine these vectors, and reshape them to get $\mathcal{Y}$.
\end{algorithmic}
\end{algorithm}

\subsection{Supervised learning}
In general, supervised learning differs from unsupervised learning primarily in how the weight matrix incorporates class label information. Supervised learning tends to outperform unsupervised learning, particularly with small datasets, due to the utilization of additional class label information.\\
In supervised learning, each data point is associated with a known class label. The weight matrix can be adapted to include this class label information. For instance, it may take the form of a block diagonal matrix, where $W_s(i) \in \mathbb{R}^{n_i \times n_i}$ represents sub-weight matrices, and $n_i$ denotes the number of data points in class $i$. Let $c(i)$ denote the class of data point $x_i$.

\textbf{Supervised PCA:} PCA is the sole linear method presented devoid of a graph matrix. Consequently, Supervised PCA implementation is not straightforward, necessitating a detailed explanation. Following the approach proposed in \cite{barshan2011supervised}, we address this challenge by formulating the problem and leveraging the empirical Hilbert-Schmidt independence criterion (HSIC):
\begin{equation*}
\begin{aligned}
\arg \max_{P^\top P=I} &
\operatorname{Tr}(P^\top X H K_L H X^\top P).
\end{aligned}
\end{equation*}
where $K_L$ is the kernel of the outcome measurements $Y$. Thus, the generalization would be to solve
\begin{equation}\label{eq:S_PCA3D}
\arg \max_{\substack{\mathcal{P} \in \mathbb{R}^{I_1 \times \ldots \times I_M \times d}\\ \mathcal{P}^\top *_M \mathcal{P}=\mathcal{I}}} \operatorname{Tr} \left(\mathcal{P}^\top *_M \mathcal{X} \times_{M+1} H K_L H *_{1} \mathcal{X}^\top *_M \mathcal{P} \right),
\end{equation}
The solution of \eqref{eq:S_PCA3D} is the largest $d$ eigen-tensors of $\mathcal{X} \times_{M+1} H K_L H \times_{1} \mathcal{X}^\top$.\\
Notice that when, $K_L=I_n$, we get the same problem as in the unsupervised case.\\
The algorithm bellow, shows the steps of the Supervised PCA via the Einstein product.

\begin{algorithm}[H]
\caption{Supervised PCA-Einstein}
\label{algorithm_PCA_Supervised}
\hspace*{\algorithmicindent} \textbf{Input:} $\mathcal{X}$ (Data) $d$(dimension output). $K_L$ (Kernel of labels)\\
\hspace*{\algorithmicindent} \textbf{Output:} $\mathcal{P}$ (Projection space).
\begin{algorithmic}[1]
\State Compute the largest $d$ eigen-tensors of $\mathcal{X} \times_{M+1} H K_L H \times_{1} \mathcal{X}^\top$.
\State Combine these tensors to get $\mathcal{P}$.
\end{algorithmic}
\end{algorithm}

\textbf{Supervised Laplacian Eigenmap:}
The supervised Laplacian Eigenmap is similar to the Laplacian Eigenmap, with the difference that the weight matrix is computed using the class label, many approaches were proposed \cite{raducanu2012supervised,costa2005classification,tai2022kernelized}. We choose a simple approach that changes the weight matrix to $W_s$, and the rest of the algorithm is the same.

\textbf{Supervised LLE:}
There are multiple variants of LLE that uses the class label to improve the performance of the method, e.g., Supervised LLE (SLLE), probabilistic SLLE \cite{zhao2009supervised}, supervised guided LLE using HSIC \cite{alvarez2011global}, enhanced SLLE \cite{zhang2009enhanced}..., the general strategy is to incorporate the class label either in computing the distance matrix, the weight matrix, or in the objective function \cite{ghojogh2020locally}. We choose the simplest which is the first strategy; By changing the distance matrix by adding term that increases the inter-class and decreases the intra-class variance. The rest of the steps are the same as the unsupervised LLE.

\subsubsection{Repulsion approaches}
In the semi-supervised or the supervised learning, how we use the class label can affect the performance, commonly, the similarity matrix, tells us only if two points are of the same class or not, without incorporating any additional information on data locality, e.g., the closeness of points of different classes.., thus the repulsion technique is used to take into account the class label information, by repulsing the points of different classes, and attracting the points of the same class. It extends the traditional graph-based methods by incorporating repulsion or discrimination elements into the graph Laplacian, learning to more distinct separation of different classes in the reduced-dimensional space by integrating the class label information directly into the graph structure. The concept of repulsion has been used in DR with different formulations \cite{zhang2006discriminant,chen2005local} before using the k-nn graph to derive it. \cite{kokiopoulou2009enhanced} a generic proposed a method that applies attraction to all points of the same class with the use of repulsion between nearby points of different classes, which was found to be significantly better than the previous approaches. Thus, we will use the same approach and generalize it to the Einstein-product.\\
The repulsion graph $\mathcal{G}^{(r)}=\{\mathcal{V}^{(r)},\mathcal{E}^{(r)} \}$ is derived from k-nn Graph $\mathcal{G}=\{\mathcal{V},\mathcal{E}\}$ based on the class label information, the weight of the edges can be computed in the simplest form as
\begin{equation*}
W_{i,j}^{(r)}=\left\{\begin{array}{ll}
1 & \text { if } (\mathbf{x}_{i},\mathbf{x}_{j}) \in \mathcal{E}, \; i \neq j , \; c(i) \neq c(j)\\
0 & \text { otherwise.}
\end{array}\right.
\end{equation*}
Hence, in the case of fully connected graph, the repulsion weight would be of the form
\begin{equation*}
W^{(r)}=\mathbf{1}_n-\operatorname{diag}(\mathbf{1}_{n_i}). 
\end{equation*}
Other weights value can be proposed.\\
The new repulsion algorithms are similar to the previous ones, with the new weight matrix $W_s=W+\beta W^{(r)}$ with $\beta$ is a parameter.

\section{Experiments}\label{sec:Exp}
To show the effectiveness of the proposed methods, we will use datasets that are commonly used in the literature. The experiments will be conducted on the GTDB dataset for the facial recognition, and the MNIST dataset for the digit recognition. We note that these datasets give the raw images instead of features. The results will be compared to the state of art methods, by using the projected data in a classifier. The baseline is also used for comparison, which is utilizing the raw data as the input of the classifier, and the recognition rate will be used as the evaluation metric for all methods. Images were chosen because the proposed methods are designed to work on multi-linear data, and the image is a typical example of such data. The proposed methods that use the multi-weight will be denoted by adding "$-MW$" to the name of the method. It is intuitive to use multi-weight for images since the third mode represent the RGB while the first two modes represent the location of the pixel.\\
The evaluation metric Recognition rate (IR) is used to evaluate the performance of the proposed method. It is defined as the number of correct classification over the total number of testing data. A correct classification is done by computing the minimum distance between the projected data training and the projected testing data. The IR is computed on the testing data.

For simplicity, we used the supervised version of methods, with Gaussian weights, and the recommended parameter in \cite{kokiopoulou2009enhanced} (half the median of data) for the Gaussian parameter. 
\begin{equation*}
IR=100 \times \dfrac{\text{Number of recognized images in a data}}{\text{Number of images in the data}}.
\end{equation*}
All computations are carried out on a laptop computer with 2.1 GHz Intel Core 7 processors 8-th Gen and 8 GB of memory using MATLAB 2021a.
\subsection{Digit recognition}
The Dataset that will be used in the experiments is the MNIST dataset \footnote{\url{https://lucidar.me/en/matlab/load-mnist-database-of-handwritten-digits-in-matlab/}}. It contains 60,000 training images and 10,000 testing images of labeled handwritten digits. The images are of size $28 \times 28$, and are normalized gray images. The evaluation metric is the same as the facial recognition. We will work with smaller subset of the data to speed up the computation. e.g., 1000 training images and 200 testing images taking randomly from the data. Observe that the multi-weight methods are not used in this data since it is gray data, thus, we don't have multiple weights.\\
Table~\ref{Tab:digit} shows the performance of different approaches compared to the state-of-art based on different subspace dimensions.
\begin{table}[H]
\centering
\begin{tabular}{c|ccccccc}
\textbf{-} & \textbf{OLPP} & \textbf{OLPP-E} & \textbf{ONPP} & \textbf{ONPP-E} & \textbf{PCA} & \textbf{PCA-E} & \textbf{Baseline} \\
\hline
5 & 50,50 & 50,50 & 56,00 & 56,00 & 63,00 & 63,00 & 8,50 \\
10 & 75,50 & 75,50 & 81,50 & 81,50 & 82,50 & 82,50 & 8,50 \\
15 & 81,00 & 81,00 & 80,50 & 80,50 & 84,50 & 84,50 & 8,50 \\
20 & 85,00 & 85,00 & 83,50 & 83,50 & 88,00 & 88,00 & 8,50 \\
25 & 86,00 & 86,00 & 87,50 & 87,50 & 88,00 & 88,00 & 8,50 \\
30 & 85,50 & 85,50 & 87,50 & 87,50 & 89,00 & 89,00 & 8,50 \\
35 & 88,00 & 88,00 & 89,50 & 89,50 & 87,00 & 87,00 & 8,50 \\
40 & 88,00 & 88,00 & 89,00 & 89,00 & 87,50 & 87,50 & 8,50 \\
\hline
\end{tabular}
\caption{Performance of methods per different subspace dimension.}
\label{Tab:digit}
\end{table}
The results are similar in the MNIST dataset between the method with its Multi dimensional counterpart, we claim that it is due to the fact that the vectorization of 2 dimension to 1 does not affect much the accuracy, which leads to similar results using the proposed parameters. 

Note that the objective is to compare a method with its proposed multi dimension counterpart via the Einstein product to see if the generalization works.
\subsection{Facial recognition}
The dataset that will be used in the experiments is the Georgia Tech database GTDB crop \footnote{\url{https://www.anefian.com/research/face_reco.htm}}. It contains 750 color JPEG images of 50 person, with each one represented by exactly 15 images that show different facial expression, scale and lighting conditions. Figure \ref{fig:GTDB}
shows an example of 12 arbitrary images from the possible 15 of an arbitrary person in the data set.\\
Our data in this case is a tensor of size $height \times width \times 3 \times 750$ when dealing with RGB, and $height \times width\times 750$ when dealing with gray images. The height and width of the images are fixed to $60 \times 60$. The data is normalized.

\begin{figure}[H]
\centering
\includegraphics[width=0.8\textwidth]{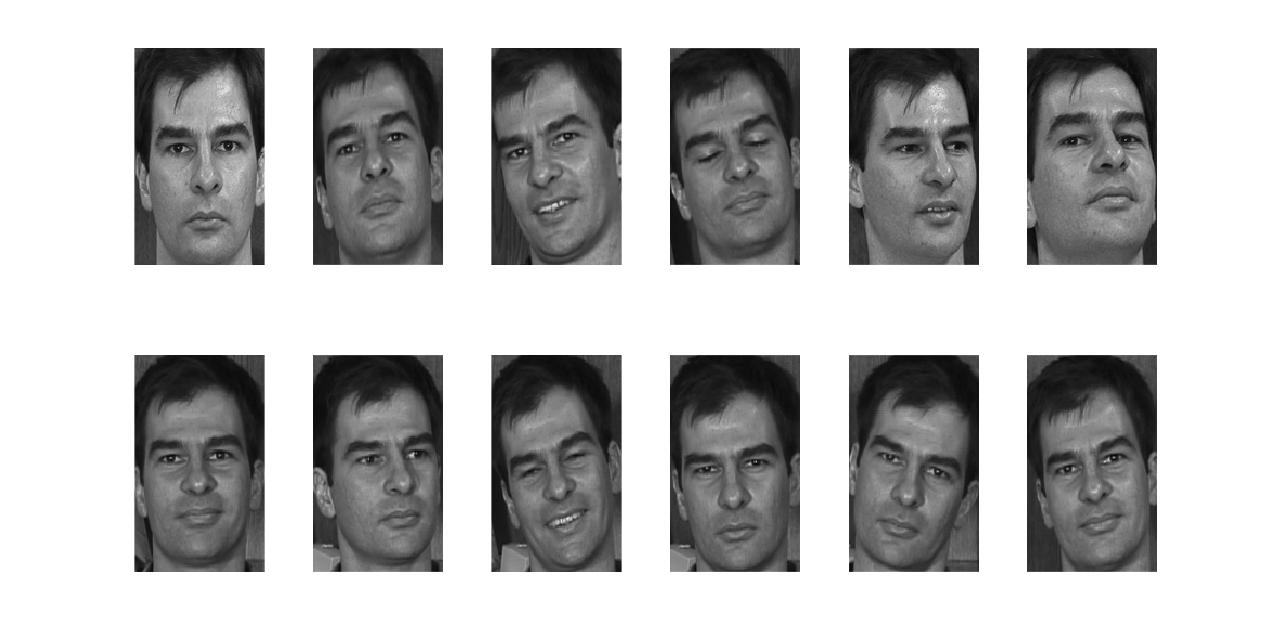}
\caption{Example of images of one person in the GTDB dataset.}
\label{fig:GTDB}
\end{figure}

The experiment is done using 12 images for training and 3 for testing per face. Figure \ref{fig:GTDB_results} shows these results for different subspace dimension reduction
\begin{figure}[H]
\centering
\includegraphics[width=0.8\textwidth]{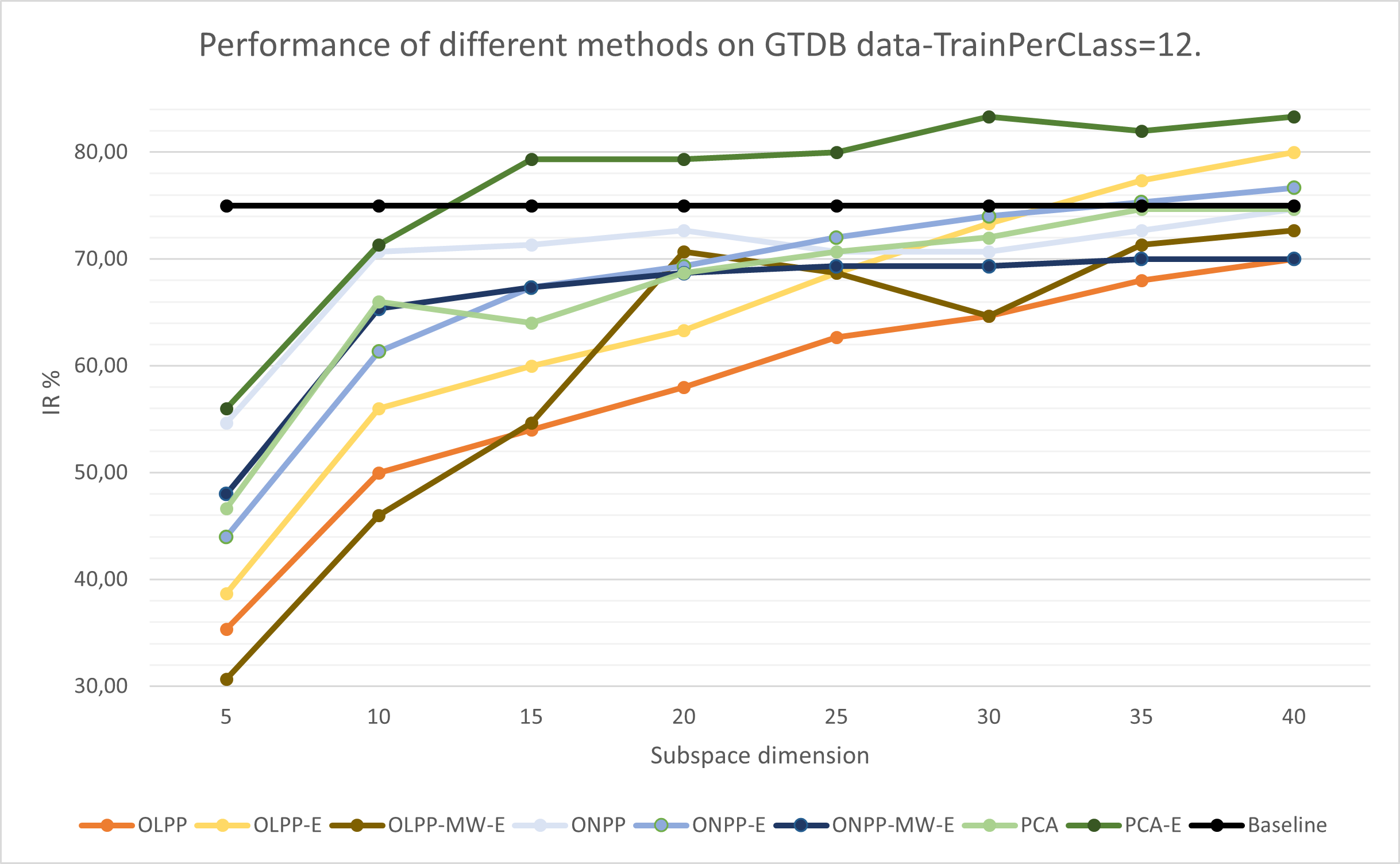}
\caption{Performance of methods on different subspace dimension.}
\label{fig:GTDB_results}
\end{figure}
The results show that as the subspace dimension increases, the performance of most methods also increases, suggesting that these methods benefit from a higher dimensional-feature space up to a point that differ from a method to another. The generalized methods using the Einstein product gives overall better result on all subspace dimension compared to its counterparts, except for the ONPP in the small $d$ case. The Multiple-weight methods show varying performance. They outperform the single-weight in some cases. Future work could be considered to enhance how the to aggregate the results of each weights in order to give a more robust results.\\
The objective is to compare between a method and its multi dimensional counter parts via the Einstein product, e.g., the OLPP method with the OLPP-E, and OLPP-E-MW.\\
The superiority of Einstein based methods can be justified by the fact that, it preserve the multi-linear structure of the data, and the non-linear structure of the data, which is not the case of the vectorization of the data, which is the case of the other matricized methods.

\section{Conclusion}\label{sec:Cnc}
The paper advances the field of dimension reduction by introducing refined graph-based methods and leveraging the Einstein product for tensor data. It extends both the Linear and Nonlinear methods (supervised and unsupervised) to higher order tensors as well as its variants. The methods are conducted on the GTDB and MNIST dataset, and the results are compared to the state-of-art-methods showing the competitive results. A future work could be conducted on generalization on trace ratio methods as Linear Discriminant Analysis. An acceleration of the computation can also be proposed using the Tensor Golub Kahan decomposition to get an approximation of these eigen-tensors in constructing the projected space.

\bibliographystyle{siam}
\bibliography{cas-refs}
\end{document}